\documentclass[
preprint, 3p, 
number, 
sort&compress,
]{elsarticle}
\pdfoutput=1

\usepackage[utf8]{luainputenc}
\usepackage[english]{babel}
\usepackage{csquotes}

\usepackage[plainpages=false,pdfpagelabels,hidelinks,unicode]{hyperref}

\usepackage{amsmath}
\allowdisplaybreaks
\usepackage{amssymb}
\usepackage{commath}
\usepackage{mathtools}
\usepackage{bbm}
\usepackage{bm}
\usepackage{siunitx}
\usepackage{adjustbox}

\usepackage{amsthm}
\theoremstyle{plain}
\newtheorem{theorem}{Theorem}[section]

\newtheorem{lemma}[theorem]{Lemma}
\newtheorem{proposition}[theorem]{Proposition}
\newtheorem{remark}[theorem]{Remark}
\newtheorem{definition}[theorem]{Definition}

\newtheorem{assumption}[theorem]{Assumption}

\usepackage{color}
\usepackage{graphicx,graphics}
\usepackage[small]{caption}
\usepackage{subcaption}

\ifx\useTikzForPlotting\undefined
\else
  \usepackage{pgfplots}
  \pgfplotsset{compat=1.11}
  \usetikzlibrary{external}
\fi

\usepackage{booktabs}
\usepackage{rotating}
\usepackage{multirow}

\usepackage{multicol}
\usepackage{enumitem}

\usepackage{calc}
\usepackage{xparse}
\usepackage{pifont}

\newcommand{\R}{\mathbb R}

\newcommand{\PP}{\mathbb P}
\newcommand{\TT}{\mathcal{T}}

\newcommand{\EE}{\mathcal{E}}
\newcommand{\ee}{\mathrm{e}}

\newcommand{\diff}[1]{{\mathrm{d}{#1}}}

\newcommand{\IdxM} {\mat{\mathbb{I}}}

\newcommand{\diag}{\mathop{\mathrm{diag}}}

\newcommand{\bbf}{{\mathbf {f}}}

\newcommand{\bn}{{\mathbf {n}}}

\newcommand{\bu}{\mathbf{u}}

\newcommand{\bx}{\mathbf{x}}

\newcommand{\ww}[1]{\underline{#1}}
\renewcommand{\div}{\operatorname{div}}
\newcommand{\est}[1]{\left\langle#1\right\rangle}
\newcommand{\bU}{\mathbf{U}}

\newcommand{\dd}{\mathrm{d}}
\newcommand{\bV}{\mathbf{V}}
\newcommand{\mean}[1]{\overline{#1}}

\newcommand{\Ol}{\mathcal{O}}

\newcommand{\VV}{\mathcal{V}}

\newcommand{\CU}{\mathcal{U}}

\renewcommand{\vec}[1]{\ww{#1}}
\newcommand{\mat}[1]{\underline{\underline{#1}}}

\usepackage{bbm}

\def\R{\mathbb{R}}

\definecolor{darkspringgreen}{rgb}{0., 0.55, 0.3}
\definecolor{dartmouthgreen}{rgb}{0.05, 0.5, 0.06}
\definecolor{etonblue}{rgb}{0.59, 0.78, 0.64}
\definecolor{airforceblue}{rgb}{0., 0.4, 0.66}
\definecolor{arylideyellow}{rgb}{0.91, 0.84, 0.42}
\definecolor{emerald}{rgb}{0.31, 0.78, 0.47}
\definecolor{uclagold}{rgb}{1.0, 0.7, 0.0}
\definecolor{cadmiumorange}{rgb}{0.93, 0.53, 0.18}

\newtheorem{con}{Conclusion}

 \newcommand{\vecfnum}{\vec{f}^\mathrm{num}}
 \newcommand{\fnum}{f^\mathrm{num}} 
 
  \newcommand{\fnumx}{f^{\mathrm{num},x}}

  \newcommand{\bfnum}{\mathbf{f}^\mathrm{num}}
 
 \newcommand{\bvu}{\vec{\mathbf{u}}}
 \newcommand{\bvf}{\vec{\mathbf{f}}}
 \newcommand{\bvecfnum}{\mathbf{\vec{f}}^\mathrm{num}}
 
 \renewcommand{\bm}{\mathbf{m}}

\newsavebox{\DelimiterBox}
\newlength{\DelimiterHeight}
\newlength{\DelimiterDepth}
\newsavebox{\ArgumentBox}
\newlength{\ArgumentHeight}
\newlength{\ArgumentDepth}
\newlength{\ResizedDelimiterHeight}
\newlength{\ResizedDelimiterDepth}

\newcommand{\encloseby}[3]{%
  \savebox{\ArgumentBox}{$\displaystyle #1$}%
  \settoheight{\ArgumentHeight}{\usebox{\ArgumentBox}}%
  \settodepth{\ArgumentDepth}{\usebox{\ArgumentBox}}%
  \savebox{\DelimiterBox}{#2}%
  \settoheight{\DelimiterHeight}{\usebox{\DelimiterBox}}%
  \settodepth{\DelimiterDepth}{\usebox{\DelimiterBox}}%
  \setlength{\ResizedDelimiterHeight}{%
    \maxof{1.2\ArgumentHeight}{\DelimiterHeight}%
  }
  \setlength{\ResizedDelimiterDepth}{%
    \maxof{1.2\ArgumentDepth}{\DelimiterDepth}%
  }
  \raisebox{-\ResizedDelimiterDepth}{%
    \resizebox{\width}{\ResizedDelimiterHeight+\ResizedDelimiterDepth}{%
      \raisebox{\DelimiterDepth}{#2}%
    }%
  }
  #1
  \raisebox{-\ResizedDelimiterDepth}{%
    \resizebox{\width}{\ResizedDelimiterHeight+\ResizedDelimiterDepth}{%
      \raisebox{\DelimiterDepth}{#3}%
    }%
  }
}

\ifluatex
  \newcommand{\armean}[1]{\encloseby{#1}{$\{\mkern-5mu\{$}{$\}\mkern-5mu\}$}}
  \newcommand{\jump}[1]{\encloseby{#1}{$[\mkern-4mu[$}{$]\mkern-4mu]$}}
\else
  \newcommand{\armean}[1]{\encloseby{#1}{$\{\mkern-6mu\{$}{$\}\mkern-6mu\}$}}
  \newcommand{\jump}[1]{\encloseby{#1}{$[\mkern-3mu[$}{$]\mkern-3mu]$}}
\fi

\newcommand{\logmean}[1]{\armean{#1}_\mathrm{log}}

\newcommand{\wwto}{\stackrel{(*)-\textbf{weakly}}{\longrightarrow}}

\usepackage{lipsum}
\makeatletter
\def\ps@pprintTitle{%
 \let\@oddhead\@empty
 \let\@evenhead\@empty
 \def\@oddfoot{}%
 \let\@evenfoot\@oddfoot}
\makeatother

\usepackage{algorithm}
\usepackage[noend]{algpseudocode}
\usepackage{xcolor}
\usepackage{bbm}

\begin{document}

\begin{frontmatter}

\title{Convergence of Discontinuous Galerkin Schemes for the Euler Equations  via Dissipative Weak Solutions}

\author[label1]{M\'aria  Luk{\'a}{\v{c}}ov{\'a}-Medvid’ov{\'a} 
}
\ead{lukacova@mathematik.uni-mainz.de}

\author[label1]{Philipp \"Offner \corref{cor1}}
\ead{mail@philippoeffner.de}

\cortext[cor1]{Corresponding author: Philipp \"Offner}
\address[label1]{Institut f\"ur Mathematik, Johannes Gutenberg Universität, 55099 Mainz, Germany.}
\begin{abstract}
  In this paper, we present convergence analysis of  high-order finite element based methods, in particular, we focus on a discontinuous Galerkin scheme using summation-by-parts operators. 
To this end, it is crucial that structure preserving properties, such as
positivity preservation and entropy inequality hold. We
demonstrate how to ensure them and prove the convergence of our multidimensional high-order DG scheme via dissipative weak solutions.
In numerical simulations, we verify our theoretical results.

\end{abstract}


\end{frontmatter}

\section{Introduction}
Since  the Euler system  of gas dynamics is in general ill-posed in the class of admissible (entropy) weak solutions an alternative concept of solutions for hyperbolic conservation laws,  \emph{measure-valued solutions} (MVS), has been developed and used in the analysis of numerical schemes. The advantage of MVSs, being advocated already by DiPerna \cite{diperna1985compensated}, is that they can be identified as a limit of oscillatory approximated sequences. 
In a series of papers, Feireisl and co-authors have introduced the concept of  dissipative measured-valued and dissipative weak solutions. They showed their existence and weak-strong uniqueness properties for  the Euler (barotropic, complete) and the Navier-Stokes systems \cite{carrillo2017weak, feireisl2019uniqueness, ghoshal2021uniqueness}.\\
 In  \cite{feireisl2019convergence} the authors have firstly 
studied  in this framework the convergence of a numerical scheme. They focused on a class of entropy stable finite volume schemes and consider the  barotropic and complete compressible Euler equations in the multidimensional case. They needed to establish suitable stability and consistency estimates and demonstrated that the Young measure generated by numerical solutions represent  a dissipative measure-valued solution of the Euler system. 
Moreover, the numerical solutions converge strongly  to the strong solution on its lifespan.  To this end, the weak-strong uniqueness principle was applied \cite{gwiazda2015weak}. 
Later, the authors constructed further suitable finite volume (FV) schemes in \cite{feireisl2019finite} and tested  them also numerically \cite{feireisl2021computing}. In \cite{ben2006direct}, the convergence of a more standard FV scheme based on the solutions of generalized Riemann problems was presented.
However, up to this point, the numerical investigation of the Euler equation in terms of dissipative (measure-valued or weak) solutions have been limited to first-order numerical methods. Our aim is to extend  previous results  to entropy stable high-order discontinuous Galerkin (flux differencing)  methods using summation-by-parts operators as described \textit{inter alia} in \cite{chen2017entropy, offner2020stability, ranocha2018generalised_2}. Extensions to the residual distribution framework \cite{abgrall2012review, abgrall2019analysis} and other FE based approaches \cite{guermond2018second, kuzmin2020monolithic} are in preparation. One main advantage of the flux differencing approach is that we can write and interpret it in a FV manner. 
Thus, we are able to transform the results of  \cite{feireisl2021numerics} to our high-order methods.  As far as we know, our convergence results  for higher order schemes are the first available in the literature. This also demonstrates well that the solution concept of dissipative weak solutions is consistent with the framework of high-order numerical methods. It is our belief that it is the most promising ansatz in the investigation of the Euler equations in the current state of the art.\\
The paper is organized as follows:
In the second section, we introduce the concept of dissipative weak solutions\footnote{For a more detailed introduction as well as dissipative measure valued solutions, we refer to \cite{feireisl2021numerics}.} for the Euler equations. We focus only on the complete Euler system in  two space dimensions, our results can be generalized to the  barotropic model and three dimensional case, too.  We will introduce a high-order DG method and show how the entropy conservation/dissipation property is ensured in Section  \ref{se:flux_differencing}. We will further demonstrate some weak BV estimate which follows from the entropy inequality. Afterwards, we investigate the consistency property of the DG method and in Section \ref{se:convergence}, we prove 
the convergence of the high-order DG method to a
dissipative weak solution. Using  the weak-strong uniqueness principle, we can ensure that 
the numerical solutions converge strongly to the strong solution on the lifespan of  latter. In numerical simulation in Section \ref{eq:numerics}, we verify our theoretical results. Conclusion, in Section \ref{se_con}, finishes this manuscript. In the Appendix, Section \ref{sec_appendix}, we introduce some notations and additional definitions for completeness. We give further technical details about the flux-differencing method and the fully-discrete setting by applying limiters.  
 
\section{Dissipative Weak Solutions for the Complete Euler System} \label{se:Foundation}
 
In this work, we focus  on  two-dimensional \textbf{complete Euler equations} describing gas dynamics and introduce for them 
dissipative weak solutions (DWS).
The Euler equations are one of the most investigated systems in computational fluid dynamics.
Derived from the conservation laws (mass conservation, momentum conservation (Newton's second law) and energy conservation (first Law of Thermodynamics)),  the Euler equations are  formulated in the conservative variables density $\rho$, momentum $\bm=\rho \bu$ and total energy 
$E=\frac{1}{2} \rho |\bu|^2 +\rho e$, where $e$ is the internal energy and $\bu:=(u_1, u_2)^T$ the velocity field. 
The equation of state for an ideal gas 
$p=(\gamma-1)\rho e$ with $\gamma>1 $ and pressure $p$ is used. 
The Euler equations are given as follows 
\begin{equation}\label{eq_Euler_conservation}
 \begin{aligned}
   \partial_t \rho +\div_x \bm=&0,\\
   \partial_t \bm +\div_x \left( \frac{\bm\otimes \bm}{\rho} \right) +(\gamma-1) \nabla_x \left(
   E-\frac{1}{2} \frac{|\bm|^2}{\rho}\right)&=0,   \Longleftrightarrow \qquad \bU+\div \mathbf{f}=0\\
   \partial_t E +\div_x \left [\left( E +(\gamma-1)\left(
   E-\frac{1}{2} \frac{|\bm|^2}{\rho}\right) \right) \frac{\bm}{\rho} \right]&=0
 \end{aligned}
\end{equation}
with $(t, \bx) \in (0,T)\otimes \Omega$. Here,  $\bU=(\rho, \bm, E)^T$ are the conserved variables and 
$\mathbf{f}_m=(\rho u_m, u_m \bm +p \mathbf{e}_m, u_m(E+p))^T, m=1,2$ are the flux functions where $\mathbf{e}_m$ represents the m-th row of the unit matrix. \\
In the whole text, we consider the bounded domain $\Omega \in \R^2$ together with periodic or no-flux boundary conditions.\\
For the Euler equations \eqref{eq_Euler_conservation} the mathematical entropy  can be given by 
\begin{equation}\label{eq:entropy_convex}
\eta=- \frac{\rho  s}{\gamma-1}
\end{equation} 
with thermodynamic entropy   $s:=\log \frac{p}{\rho^{\gamma}}$.
The corresponding entropy flux $\mathbf{g}:=(g_1, g_2)$ is defined by $g_m=\eta \cdot u_m$, $m=1,2$, with the velocity vector $\bu$. We obtain the entropy variable 
\begin{equation}\label{eq:entropy_variables}
\mathbf{w}=\eta'(\bU)= \left(\frac{\gamma}{\gamma-1}-\frac{s}{\gamma-1}-\frac{\rho |\mathbf{\bu}|^2}{2p}, \frac{\rho u_1}{p},\frac{\rho u_2}{p},
-\frac{\rho}{p} \right)^T
\end{equation}
and entropy potential $\Psi=\rho \bu$.
Additionally to \eqref{eq_Euler_conservation}, we require the following entropy inequality
 \begin{equation}\label{iq:entropy}
 \frac{\partial}{\partial t} \eta +\div_x \bf{g} \leq 0.
 \end{equation}

In this work, we focus on the convergence properties of higher order DG methods to 
\textbf{dissipative weak (DW) solution for the Euler equations}. Through the investigation of DeLellis and Sz{\'e}kelyhidi
that weak entropy solutions of the Euler equations are not unique\footnote{Already in \cite{sever1989uniqueness, sever1990correction}  non-uniqueness of weak entropy solutions has been presented for a constructed hyperbolic system of conservation laws.  }  \cite{de2010admissibility},
a lot of further attention  has been  given to the concept of measure-valued  solutions
\cite{brenier2011weak, gwiazda2015weak} in the context of Euler equations. 
DW solutions can be seen as a natural closure of a set of consistent approximations with respect to a weak topology. 
They fulfill the Euler equations up to the defect measures associated to possible concentrations and oscillations. 
Due to this fact, the class is large enough to contain all possible limits of consistent and stable numerical schemes. 
Therefore, DW solutions can be interpreted  as space-time expected values or barycenters of the associated dissipative measure-valued solutions.
For  the barotropic Euler equations, the definition of DW solution does not contain any Young measure and 
is therefore elegant.  However, in this work, we are focusing on the complete Euler system and we have the explicit appearance of the Young measure in the convective term in the entropy balance. 
We consider the Euler equations  with periodic boundary conditions. Alternatively, impermeable boundary conditions can also be considered, cf. \cite{feireisl2021numerics}. 
For the definition, we need also the following notations from \cite{feireisl2021numerics}.
We denote by $\mathcal{M}^+(\overline{\Omega})$  the set of all nonnegative Borel measures on a topological space $\overline{\Omega}$. 
With $\mathcal{M}(\overline{\Omega})$ the set of all signed \textbf{Radon measures} is described. They  can be identified at the space of all linear forms on $C_c(\overline{\Omega})$, especially if $\overline{\Omega}$ is compact,
i.e. $[C_c(\overline{\Omega})]^*=\mathcal{M}(\overline{\Omega})$. Finally,  the set of  \emph{positive semi-definite matrix valued measures} $\mathcal{M}^+(\overline{\Omega}; \R^{d\times d}_{sym})$ is defined as follows:   
\begin{equation*}
\mathcal{M}^+(\overline{\Omega}, \R^{d\times d}_{sym}) = \left\{ \nu \in \mathcal{M}^+(\overline{\Omega}, \R^{d\times d}_{sym}) \big|
\int_{\overline{\Omega}} \phi(\xi \otimes\xi): \diff \nu\geq0  \text{ for any } \xi \in \R^d, \phi \in C_c(\overline{\Omega}), \phi\geq 0 
 \right\}.
\end{equation*}

\begin{definition}[Dissipative Weak Solution for the Euler Equations]\label{def_dmv}
Let $\Omega\subset \R^2$ be a bounded domain. 
We call $ [\rho, \bm, \eta]$ a \textbf{dissipative weak (DW) solution}  of the complete Euler system
with periodic conditions, and the initial condition $[\rho_0, \bm_0, \eta_0]$ with $\rho>0$ and $\int_{\Omega} \frac{1}{2} 
\frac{|\bm_0|^2}{\rho_0} + e(\rho_0, \eta_0)\diff \bx <\infty$ if the following holds:

\begin{itemize}
\item The functions are \textbf{weakly continuous}: $\rho  \in C_{weak}([0,T];L^{\gamma}(\Omega) )$, 
$\bm  \in C_{weak}([0,T];L^{\frac{2\gamma}{\gamma+1}}(\Omega;\R^2) )$, $\eta \in  L^{\infty}(0, T; L^{\gamma}(\Omega)) \cap BV_{weak}([0,T];L^{\gamma}(\Omega)).$ 
\item  A measure $\mathfrak{E} \in  L^{\infty}(0,T;\mathcal{M}^+(\overline{\Omega}))$ (energy defect), exists, such that the \textbf{energy inequality} 
\begin{equation*}
 \int_{\Omega} \left[   \frac{1}{2} \frac{|\bm|^2}{\rho} + \rho e(\rho,\eta )\right]  ( \tau,\cdot) \diff \bx + \int_{\Omega}   \diff{\mathfrak{E}}(\tau)
 \leq  \int_{\Omega} \left[ \frac{1}{2} \frac{|\bm_0|^2}{\rho_0} +\rho_0 e(\rho_0,\eta_0) \right] \diff \bx
\end{equation*} 
is fulfilled for a.a.  $0\leq \tau\leq T$.
\item  The weak \textbf{equation of continuity}
\begin{equation*}
\left[  \int_{\Omega} \rho \varphi \diff \bx \right]_{t=0}^{t=\tau}= \int_{0}^\tau \int_{\Omega} 
\left[ \rho \partial_t \varphi + \bm \cdot \nabla_{\bx} \varphi  \right] \diff \bx \diff t
\end{equation*} 
is satisfied for any $0 \leq \tau \leq T$  and any $\varphi \in C_c^\infty((0,T)\times \Omega)$. 
\item  Let $
 \mathfrak{R}\in L^{\infty} \left(0,T; \mathcal{M}\left(\overline{\Omega}, \R^{d\times d}_{sym}\right) \right) 
$ be the Reynolds defect. The integral identity derived from the \textbf{momentum equation}
 \begin{align*}
 \left[  \int_{\Omega} \bm \cdot \mathbf{\varphi} \diff \bx\right]_{t=0}^{t=\tau}
 = &\int_{0}^\tau \int_{\Omega} 
 \left[ \bm \cdot \partial_t \mathbf{\varphi }+ 
 1_{\rho>0} \frac{ \bm \otimes  \bm }{\rho}   : \nabla_{\bx} \mathbf{\varphi} +1_{\rho>0} p(\rho, \eta)  \div_{\bx} \mathbf{\varphi} \right]
 \diff \bx \diff t 
   + \int_{0}^\tau \int_{\Omega} \nabla_{\bx} \mathbf{\varphi}: \diff{\mathfrak{R}}
 \end{align*}
holds for any $0 \leq \tau\leq t$ and any test function $  \mathbf{\varphi} \in C^\infty([0,T]\times \overline{\Omega};\R^d)$.
 \item  The  weak \textbf{entropy  inequality} 
\begin{equation*}
\begin{aligned}
\left[\int_{\Omega} \eta \varphi  \diff \bx \right]_{t=\tau_1-}^{t=\tau_2+} 
\leq&  \int_{\tau_1}^{\tau_2} \int_{\Omega} \left[ 
\eta \partial_t \varphi +\est{\nu; 1_{\tilde{\rho}>0} \left( \tilde{\eta} \frac{\tilde{\bm}}{\tilde{\rho}} \right) } \cdot \nabla_\bx \varphi  \right] \diff \bx \diff t\\
\eta(0-, \cdot) =& \eta_0
\end{aligned}
\end{equation*}
exists for any $0\leq \tau_1 \leq \tau_2 <T$, 
any $\varphi \in C_c^\infty((0,T)\times \Omega), \varphi\geq 0$, where $\{ \nu_{t,\bx}\}_{(t, \bx) \in (0,T) \times \Omega}$ is a 
parametrized (Young) measure
\begin{equation}\label{eq:Young}
\begin{aligned}
\nu \in L^{\infty}((0,T)\times \Omega;\mathcal{P}(\mathcal{F})), \mathcal{F}=\left\{ \tilde{\rho} \in \R, \tilde{\bm}\in \R^d, \tilde{\eta}\in \R  \right\};\\
\est{\nu, \tilde{\rho}} =\rho, \est{\nu, \tilde{\bm}} =\bm, \est{\nu, \tilde{\eta}} =\eta,\\
\nu_{t, \bx}\left\{ \tilde{\rho} \geq 0, (1-\gamma)\tilde{\eta} \geq \underline{s} \tilde{\rho} \right\}=1 \text{ for a. a. } (t,\bx) \in (0,T)\times \Omega; 
\end{aligned}
\end{equation} 
\item For some constants $0\leq c_1 \leq c_2,$ and 
$
\mathfrak{E} \geq \est{\nu; \frac{1}{2} \frac{|\tilde{\bm}|^2 }{\tilde{\rho}} + \tilde{\rho} e(\tilde{\rho}, \tilde{\eta})} - \left( \frac{1}{2} \frac{|\bm|^2}{\rho} +\rho e(\rho, \eta)\right),
$
the  \textbf{defect compatibility conditions}
$
c_1 \mathfrak{E} \leq \operatorname{tr} [\mathfrak{R}] \leq c_2 \mathfrak{E}
$
holds.
\end{itemize}
\end{definition}

From the weak formulation, one can eliminate completely the energy defect by setting $\mathfrak{E}=\operatorname{tr} [\mathfrak{R}]$ (modulo a multiplicative constant).
Both defects include the concentration and oscillation defects which are separately used in the definition of dissipative measure-valued (DMV) solutions. It can be shown, cf. \cite{feireisl2021numerics}, that following relation holds
\begin{equation*}
\begin{aligned}
 \mathfrak{R}= &\overline{\left[ 1_{\rho>0} \frac{\bm\otimes \bm }{\rho } +p(\rho) \mathbb{I}  \right] } -\est{\nu, 1_{\tilde{\rho}>0} \frac{\tilde{\bm}\otimes \tilde{\bm} }{ \tilde{\rho} } +p(\tilde{\tilde{\rho}}) \mathbb{I}  } \\
 +& 
 \est{\nu, 1_{\tilde{\rho}>0} \frac{\tilde{\bm}\otimes \tilde{\bm} }{ \tilde{\rho} } +p(\tilde{\tilde{\rho}}) \mathbb{I}  } 
 -
 \left(1_{\rho>0} \frac{\bm\otimes \bm }{\rho } +p(\rho) \mathbb{I}  \right),
 \end{aligned}
\end{equation*}
where the first two terms denote the concentration defect and the last two terms are called oscillation defect. Here, the overline term represent the weak limit of a sequence of approximated solutions $\CU^h=[\rho^h, \mathbf{m}^h, \eta^h]$ which will be generated later by our our consistent DG scheme, i.e. 
\begin{equation*}
  \frac{\bm^h\otimes \bm^h }{\rho^h } +p(\rho^h) \mathbb{I} \longrightarrow
  \overline{ \frac{\bm\otimes \bm }{\rho } +p(\rho) \mathbb{I} } \quad \text{ weakly}-(*) \text{ in } \mathcal{M}\left(\overline{\Omega}, \R^{d\times d}_{sym}\right),
\end{equation*}
where $\{ \nu_{t,\bx}\}_{(t, \bx) \in (0,T) \times \Omega}$ is the Young measure generated by $\CU^h=[\rho^h, \mathbf{m}^h, \eta^h]$.
\begin{remark}[Dissipative Weak Solutions versus Dissipative Measure-Valued Solution]\label{eq:Remark_DMS_DW}
As an alternative to dissipative weak solutions, one may apply \emph{dissipative measure-valued solutions}, cf. \cite[Definition 5.3]{feireisl2021numerics} which give a more detailed description. Only few differences can be recognized:
\begin{enumerate}
\item The time mapping $t\to \nu_{t, \bx} \in L^{\infty}(\Omega, \mathcal{P}(\mathcal{F}))$
is \textit{a priori} not continuous.
However, its barycenter
coordinates $t\to \est{ \nu_{t, \bx}, \tilde{\rho}} \in C_{weak}([0,T]; L^{\gamma}(\Omega))$ is continuos (similar for momentum and entropy).
\item 
As it is observed in \cite[Section 5.2]{feireisl2021numerics}, the only relevant piece of information of the limit process 
 is the value of some \emph{observables} like the barycenter of the measure. It converges strongly in a suitable norm.
Note that the barycenter of any DMV solution of the Euler system represents a DW solution. 
\end{enumerate}

\end{remark}

In case when a classical solution 
exists our generalized solution  coincides with the classical one. One speaks about \textbf{compatibility property}. 
In our framework, compatibility means that the defect measures used in the Definition 
\ref{def_dmv}  vanish as long as the DW solution is smooth enough. 
As it is shown in \cite{feireisl2021numerics}  if $[\rho, \bm, \eta]$ belongs to 
\begin{equation}\label{eq_regularity}
\rho \in C^1([0,T]\times \overline{\Omega}), \; \inf_{(0,T)\times \Omega} \rho>0, \; \mathbf{u} \in C^1([0,T]\times \overline{\Omega}; \R^d), \; \eta \in C^1([0,T]\times \overline{\Omega})
\end{equation}
then $[\rho, \bm, \eta]$ is a classical solution of the complete Euler system.
Moreover, see \cite[Theorem 5.7.]{feireisl2021numerics}.

\section{Discontinuous Galerkin Schemes} \label{se:flux_differencing}
In the following part, we will shortly introduce the considered DG methods and the relevant notations following \cite{pazner2021sparse}. Throughout the paper, we will confine ourselves to the semidiscrete scheme meaning that
the time will remain continuous, the discretization is done only in space. By using implicit methods, our investigation can be further developed to the fully discrete setting. However, explicit methods are the natural choice for solving hyperbolic conservation laws and we also use explicit ones in our numerical simulations in Section \ref{eq:numerics}. In Appendix \ref{se:extension}, we describe how we ensure practically the positivity of density and pressure as well as  the entropy inequality in the fully discrete setting.  

\subsubsection*{Notations}
The spatial domain $\Omega \subset \R^2$ is discretized  with a mesh of tensor-product elements, e.g. regular quadrilateral grid,  denoted by $\TT_h$. We denote the generic cell $K$ and the uniform mesh size with $h$. They  are given by 
$$K:= [x_{i-1/2,j}, x_{i+1/2,j}] \times [y_{i,j-1/2}, y_{i,j+1/2}]$$
with
$h:= x_{i+1/2,j}-x_{i-1/2,j}=y_{i,j+1/2}-y_{i,j-1/2}$, for simplicity. 
Extensions to  (unstructured) rectangular meshes with cell sizes $h_x\neq h_y$ are straightforward. For triangular grids we may follow the approach presented in \cite{chenreview}.
With $\partial K$ we denote the boundary of an element $K$ and by $\EE$ the set of all interfaces of all cells $K\in \TT_h$ where $\ee$ is one interface of $\partial K$. Between two elements $K^-$ and $K^+$ we have a normal vector $\bn$. We have  normal vector given by either $\bn=(n_x, 0)$ or $\bn=(0,n_y)$ depending on the interface. 
Let $\mathcal{Q}^p([-1,1]^2)$  be the space of all multivariate 
polynomials of degree at most $p$ in each variable.
On each element $K\in \TT_h$, we have a linear map  $T_K:[-1,1]^2\to K$  and    $\mathcal{Q}^p(K)$  is  spanned by functions $\phi \circ T_K^{-1}$.
The DG solution space $\VV^h$ is given by
 \begin{equation}\label{eq:space}
   \VV^h = \left\{ v^h\in L^1(\Omega)  \Big| v^h := v^h|_{K} \in \mathcal{Q}^p(K) \text{ for all } K\in \TT_h \right\}.
 \end{equation}
Approximated solutions to the Euler equations \eqref{eq_Euler_conservation} live in  the vector version of this 
 space  $\tilde{\VV}^h= [\VV^h]^4$. 
 To describe an element of the finite dimensional space $\VV^h$, we apply a nodal Gauss-Lobatto basis. We denote by  $\xi_i$ the Gauss-Lobatto points in the interval $[-1,1]$. Further,  $L_i$ is the Lagrange polynomial which fulfills $L_i(\xi_j)=\delta_{i,j}$. Here,  $\delta_{i,j}$ is the Kronecker delta. Lagrange polynomials form a basis  for $\mathcal{Q}^p([-1,1])$ in one dimension. We obtain a basis for
 $\mathcal{Q}^p([-1,1]^2)$  via the tensor product of the one-dimensional basis, i.e. $L_i(x)L_j(y)$.
 In each element, each component of the conservative variable vector of the Euler equation  is approximated by a polynomial in the reference space \eqref{eq:space}. The nodal values (interpolation points) are our degrees of freedom (DOFs) where we have to calculate the time-dependent nodal coefficients for all components (density, momentum, energy) in the following.
  To make this point clear, our numerical solution  is given by $\bU^h=(\rho^h, \bm^h, E^h)^T \in \tilde{\VV}^h$ 
  and each component is represented by a polynomial, e.g. in the reference element the approximated density is given by 
  \begin{equation}\label{eq:approx}
  \rho^h(x,y,t) = \sum_{i,j=0}^p \hat{\rho}^h(\xi_i, \nu_j, t) L_i(x) L_j(y)= \sum_{i,j=0}^p \hat{\rho}_{i,j}(t) L_i(x) L_j(y)
  \end{equation}
where $ \xi$ and $\nu$ are our Gauss-Lobatto nodes in $x-$ and $y-$direction
and 
$\hat{\rho}_{i,j}(t)$ is the time-dependent coefficient of our polynomial presentation, i.e. $\hat{\rho}_{i,j}(t)= \rho^h(\xi_i, \nu_j,t)$. 
Analogous notation holds for the other components of $\bU^h$.

%
%

\subsubsection*{DG formulation}

 We begin by defining the high-order DG discretization for the Euler equation \eqref{eq_Euler_conservation}. 
First, we multiply by a  test function $\bV^h \in \tilde{\VV}^h$ and integrate over the domain $\Omega$ \eqref{eq_Euler_conservation}.
By integrating twice the term with the space derivatives, we obtain the strong DG form\footnote{Note that at the continuous level, the strong and the weak DG forms are equivalent for a sufficiently smooth solution.}
 \begin{equation}\label{eq:DG_strong_analytical}
  \int_{\Omega} \partial_t \bU^h \cdot \bV^h \diff \bx +\sum_{K\in \TT_h} \int_K
  (\div \mathbf{f} (\bU^h) )\cdot \bV^h \diff \bx + \sum_{\partial K^- \in \TT_h} \int_{\partial K^-} (\bfnum(\bU^{h,-}, \bU^{h,+},\bn^-) - \mathbf{f} \cdot \bn^-) \cdot  \bV^{h,-}\diff s =0 
 \end{equation}
where $\mathbf{f}(\bU^h)=(\mathbf{f}_1,\mathbf{f}_2)$ with $\mathbf{f}_m$ as defined in \eqref{eq_Euler_conservation} and the numerical flux $\bfnum$ will be specified later in \eqref{eq:LLF}.  
Note that  a dot, i.e.  $\cdot$, product means that \eqref{eq:DG_strong_analytical}
 has to be solved for each component of the Euler equation separately. 
Later all the calculations in our high-order DG method are done in a reference element $I=[-1,1]^2$.
To evaluate the integrals, we proceed by choosing the Gauss-Lobatto nodes as collocated quadrature points. 
Finally, by this selection of a tensor-product basis and Gauss-Lobatto quadrature, the DG operators have a Kronecker-product structure as well. We denote by $\mat{M}_1$ on $[-1,1]$ the one-dimensional mass matrix.
It has diagonal form with quadrature weight on the diagonal.  Further, we obtain the one-dimensional differentiation matrix $\mat{D}_{1}$  by evaluating the derivatives of the basis functions at the nodal points.  We use the index here to clarify that we have the one-dimensional setting and working on the reference element $I$. 
To clarify the setting, let $\xi_j$ be the Gauss-Lobatto quadrature points $-1=\xi_0<\xi_1<\cdots<\xi_p=1$ in $[-1,1]$
with corresponding quadrature weights 
 $\{\omega_j\}_{j=0}^p$. The nodal Lagrangian basis is given by  $L_j(\xi_l)=\delta_{jl}$ and we can define 
the discrete inner $\est{u,v}_{\omega}:=\sum\limits_{j=0}^p \omega_j u(\xi_j) v(\xi_j)$ in one space-dimension.  Then, the above described operators are given by 
 \begin{itemize}
\item Difference matrix $\mat{D}_1$ with $\mat{D}_{1,jl} =L_l'(\xi_j)$
\item Mass matrix $\mat{M}_{1,jl}=\est{L_j,L_l}_{\omega} =\omega_j \delta_{jl}$, so that
   $\mat{M}_1 = \diag \{\omega_0, \dots, \omega_p\}$
\end{itemize}
Later, we need also the operators 
\begin{itemize}
\item Stiffness matrix $\mat{Q}_{1,jl}=\est{L_j',L_l}_{\omega} =\est{L_j,L_l'}_{\omega}$
\item Interface matrix $\mat{B}_1=\diag(-1,0,\cdots,0,1)$
\end{itemize}
Up to this point, these operators would work separately on every component of the conservative variable vector $\bU^h$, i.e. 
on the density $\rho^h$, momentum  $\bm^h$, and energy $E^h$. To extend these operators to the Euler system, a simple Kronecker product can be used. It is 
 \begin{equation*}
   \mat{\mathbf{M}}_1=\mat{M}_1 \otimes \IdxM_4 \qquad \mat{\mathbf{D}}_1=\mat{D}_1 \otimes \IdxM_4,
 \end{equation*}
where $ \IdxM$ is the $4\times4$-identy matrix.\\
Similar in two-dimension, we obtain the local mass matrix and differentiation matrix in the standard element $I$ through Kronecker products:
  $\mat{\mathbf{M}}_I= \mat{\mathbf{M}}_1\otimes  \mat{\mathbf{M}}_1,
  \mat{\mathbf{D}}_{1,I}= \IdxM\otimes \mat{\mathbf{D}}_1,
  \mat{\mathbf{D}}_{2,I}= \mat{\mathbf{D}}_1 \otimes \IdxM.$
To solve \eqref{eq:DG_strong_analytical}, we have to  evaluate the 
cell interface integrals. Due to the tensor structure ansatz, we evaluate at each interface the one-dimensional Gauss-Lobatto quadrature rule. Due to the Kronecker ansatz, we obtain the
following interface operators 
$
\mat{\mathbf{B}}_{\ee_1}=  \mat{\mathbf{M}}_1 \otimes \mat{\mathbf{B}}_1 \text{ and }  \mat{\mathbf{B}}_{\ee_2}= \mat{\mathbf{B}}_1 \otimes  \mat{\mathbf{M}}_1 
$
depending on the considered cell interfaces. The defined operators fulfill the summation-by-parts property meaning that they mimic discretely integration-by-parts. In Appendix \ref{eq:SBP property},  we also repeat the main SBP properties for completeness, cf. \cite{chenreview,offner2020stability}. \\
In  \eqref{eq:DG_strong_analytical}, we have to calculate  the time-dependent coefficients of our polynomial representation \\
$\bU^h= (\rho^h, \bm^h, E^h) \in\tilde{\VV}^h.$ We have to distinguish in numbering between the interpolation (quadrature) points in $x$- and $y$- direction in \eqref{eq:approx}. For simplicity, we are renumbering the points. We have $n_p=(p+1)^2$ quadrature/interpolation points\footnote{Alternative a multi-index can be used of the notation.} denoted by $\mathbf{\xi}_i$ in each element and $n_{b}=4p$ at the interfaces. On each face we have $p+1$ Gauss-Lobatto quadrature nodes but on the corners they intersect. The basis is given by $\mathbf{L}_i$ with $i \in \{1,n_p\}$.
 We denote by $\bvu$ the vector of coefficients (i.e. nodal values) of $\bU^h$ on $K$:
\begin{equation}\label{eq:coefficients}
\bvu= \left(\rho^h(\mathbf{\xi}_1), \bm^h(\mathbf{\xi}_1), E^h(\mathbf{\xi}_1 );
\rho^h(\mathbf{\xi}_2), \bm^h(\mathbf{\xi}_2),  E^h(\mathbf{\xi}_2 ); \dots ;\rho^h(\mathbf{\xi}_{n_p}), \bm^h(\mathbf{\xi}_{n_p}), E^h(\mathbf{\xi}_{n_p}) \right)^T.
\end{equation}
Let $\bvf_m$ ($m=1,2$) denote the vector of values of $\mathbf{f}_m(\bU^h)$ evaluated at the nodal points. For each cell interface $\ee \in \partial K \subset \EE$, we have to evaluate $
\mathbf{f}(\bU^h)$ at the one-dimensional Gauss-Lobatto nodes on the cell interface $\ee$ (where the trace of $\bU^h$ is taken from inside $K$ dotted with the scaled with normal vector $\bn$ facing outwards from $\ee$). We denote this by $\mat{\mathbf{R}}_{\ee_m}\bvf_{\ee_m}$. Likewise $ \bvecfnum_{\ee_m}$ denotes the nodal values of $\bfnum(\bU^{h,-}, \bU^{h,+},\bn^-)$.
With these operators, we can finally re-write the DG semidiscretization  \eqref{eq:DG_strong_analytical} on the reference element as follows 
\begin{equation}\label{eq:DG_standard}
 \partial_t \bvu +  \mat{\mathbf{D}}_{1,I}  \bvf_1 +\mat{\mathbf{D}}_{2,I}  \bvf_2= \mat{\mathbf{M}}_I^{-1}  \sum_{j \in \partial I} \mat{\mathbf{B}}_{j} \left(\mat{\mathbf{R}}_{j}\bvf_{j}-\bvecfnum_j \right),
\end{equation}
where $\partial I$ denotes the cell interfaces of the reference element.

 \subsubsection*{Entropy Stable DG Method}
 
Equation \eqref{eq:DG_standard} describes the classical discontinuous Galerkin spectral element method (DGSEM) in two-space dimension. The method is by construction not  entropy conservative/ dissipative. To obtain an high-order entropy dissipative DG method for the Euler equation, we apply the flux differencing approach. To this end we replace the volume flux  $\sum_{m=1}^2 \mat{\mathbf{D}}_{m,I}  \bvf_m $  in above equation \eqref{eq:DG_standard} using consistent, symmetric two-point  numerical fluxes. The resulting DG scheme in the reference element reads than 
 \begin{equation}\label{eq:DG_standard_2}
 \partial_t \bvu + 2 \left( \mat{\mathbf{D}}_{1,I} \bvecfnum_{1,Vol}(\bvu, \bvu)+\mat{\mathbf{D}}_{2,I} \bvecfnum_{2,Vol}(\bvu, \bvu) \right)= \mat{\mathbf{M}}_I^{-1}  \sum_{j \in \partial I} \mat{\mathbf{B}}_{j} \left(\mat{\mathbf{R}}_{j}\bvf_{j}-\bvecfnum_j \right),
\end{equation}
where $\bvecfnum_{m,Vol}$ denotes the numerical volume flux working on each degree of freedom and $\bvecfnum_j $ is the classical numerical flux at the interface. 
Alternatively, we can rewrite \eqref{eq:DG_standard_2} for each nodal value, i.e. $\bU_j^h:=(\rho^h(\xi_j), \bm^h(\xi_j), E^h(\xi_j))^T$, separately.  
\begin{equation}\label{eq:Shu_modified_two_nodal}
\begin{aligned}
 \frac{\diff{}}{\diff t} \bU_j^h + 2 \sum_{l=1}^{n_p} \left( \mathbf{D}_{1,jl} \mathbf{f}^{\operatorname{num}}_{1,Vol}(\bU_j^h,\bU_l^h) 
+ \mathbf{D}_{2,jl} \mathbf{f}^{\operatorname{num}}_{2,Vol}(\bU_j^h,\bU_l^h) \right)
=& \sum_{l=1}^{n_p} \frac{\tau_l}{\mathbf{\omega}_j} \left( \mathbf{f}_l  -\bfnum_j \right)\\
=& \frac{1}{\mathbf{\omega}_j} \left(\tau_{1,j} \mathbf{f}_{1,j}+ \tau_{2,j} \mathbf{f}_{2,j}- \tau_j \bfnum_j  \right)
\end{aligned}
\end{equation}
with $\tau_l, \tau_{1,j}, \tau_{2,j}= -1, 0, 1$ depending on the considered inteface\footnote{They are zero for internal quadrature points and $-1$ and $1$ at the corresponding interfaces} and $n_p$ are the quadrature points for the volume term. 

\subsubsection*{Numerical Fluxes}

An important aspect of the flux differencing approach 
\eqref{eq:Shu_modified_two_nodal} is the selection of numerical flux functions. In this manuscript, 
we use the local  Lax-Friedrich numerical flux function at the cell interfaces $\mathbf{f}^{\operatorname{num}}$ for simplicity.
The reason for this selection is that the Lax-Friedrich flux in its first order discretization is entropy stable\footnote{It is actually invariant domain preserving meaning the approximated solution remains in the physical meaningful domains, cf. \cite{abgrall2019reinterpretation, kuzmin2020monolithic}.}. The local Lax-Friedrich flux is given by 
\begin{equation}\label{eq:LLF}
 \fnum(\bU^{h,-}, \bU^{h,+}, \bn^-):= \frac{1}{2} \left(\mathbf{f}(\bU^{h,-}) + \mathbf{f}(\bU^{h,+}) \right)\cdot \bn^- - \frac{\lambda}{2} \left( \mathbf{f}(\bU^{h,+})-\mathbf{f}(\bU^{h,-}) \right),
\end{equation}
where $\lambda \geq \lambda_{max}$ is an upper bound for the maximum wave speed. For the Euler equations \eqref{eq_Euler_conservation}, it holds that 
$\lambda_{max}=\max\{ |\bu^h_{K^-}|+c_{K^-},|\bu^h_{K^+}|+c_{K^+} \}$ with $c=\sqrt{\gamma \frac{p}{\rho} }$.\\
For the numerical volume flux $\mathbf{f}^{\operatorname{num}}_{m,Vol}$ with $m=1,2$ , we select the consistent, symmetric and entropy conservative two-point flux of Ranocha \cite{ranocha2018generalised_2}. It is defined for each component separately: 
 \begin{equation}\label{eq:Ranocha_flux} 
 \begin{aligned}
  \fnum_{\rho,1} &=\logmean{\rho}\armean{u_1}, \quad  \fnum_{\rho u_1,1}  =\armean{ u_1} \fnum_{\rho,1} +\armean{ p}, \quad  \fnum_{\rho u_2,1}=\armean{ u_2} \fnum_{\rho} \\
  \fnumx_{E,1}&= \left(  \logmean{ \rho}  \left( \armean{ u_1}^2+\armean{ u_2}^2 - \frac{\armean{ u_1+u_2}^2}{2} \right) -\frac{1}{\gamma-1} \frac{\logmean{ \rho} }{\logmean{ \rho/p} } +
  \armean{ p} \right) \armean{u_1} \\
  &-\frac{\jump{p} \jump{\bu}}{4},
 \end{aligned}
 \end{equation}
 with $\mathbf{f}^{\operatorname{num}}_2$ defined analogously. Here, we have used the abbreviations $\armean{\rho}=\frac{\rho^{+}+\rho^-}{2} $ and
 $\logmean{\rho}=\frac{\rho^{+}-\rho^-}{\log \rho^+-\log \rho^-}$.  The chain rule for the logarithmic mean is given by
$
  \jump{\log a} = \frac{\jump{a}}{\logmean{a}}.
$
 The first three numerical fluxes are fixed whereas the flux $ \fnum_{E,1}$
 has been calculated using the condition $\jump{\mathbf{w}^T}\cdot \bfnum_1-\jump{\psi_1}=0$ with potential $\psi$. 
 It has been proven in \cite{ranocha2018generalised_2} that the numerical flux  \eqref{eq:Ranocha_flux} is entropy conservative (EC) and kinetic energy preserving (KEP)\footnote{Meaning that for a smooth solution the kinetic energy fulfills an additional conservation law.}. 
 We select those two fluxes on purpose since the properties of the 
scheme \eqref{eq:DG_standard_2} highly depend on the chosen fluxes as the following theorem confirms:

\begin{theorem}[Flux Differencing Theorem \cite{chenreview}]\label{th:shu_chen_2}
Assume that $\bvecfnum_{1,Vol}$ and $\bvecfnum_{2,Vol}$ are symmetric and consistent, and that $\bvecfnum$ is conservative and consistent. Then,
the flux differencing scheme \eqref{eq:Shu_modified_two_nodal}  is conservative and high-order accurate. If we further assume that 
both fluxes $\bvecfnum_{1,S}$ and $\bvecfnum_{2,S}$ are entropy conservative, and that $\mathbf{\vecfnum}$ is entropy stable, 
 \eqref{eq:Shu_modified_two_nodal} is entropy conservative within single elements and entropy stable across interfaces. 
\end{theorem}
\begin{proof}

First, we have by using the SBP property 
{\small
\begin{equation}\label{eq:conservation}
 \begin{aligned}
&\int_{K} \partial_t \bU^h \diff \bx=
 \frac{\diff{}}{\diff t} \sum_{j=1}^{n_p} \omega_j \bU_j^h \\
 =&
 \sum_{j=1}^{n_p} \tau_j  (\mathbf{f}_j- \bfnum_j)
 -2 \sum_{j,l=1}^{n_p} 
 \left( \mathbf{M}_{jj} \mathbf{D}_{1,jl} \mathbf{f}^{\operatorname{num}}_{1,Vol}(\bU_j^h,\bU_l^h) 
+ \mathbf{M}_{jj}  \mathbf{D}_{2,jl} \mathbf{f}^{\operatorname{num}}_{2,Vol}(\bU_j^h,\bU_l^h) \right)\\
=& \sum_{j=1}^{n_p} \tau_j  (\mathbf{f}_j- \bfnum_j)
-\sum_{j,l=1}^{n_p} 
  \left( \left( \mathbf{M}_{jj} \mathbf{D}_{1,jl} +  \mathbf{D}_{1,lj} \mathbf{M}_{jj}  \right) \mathbf{f}^{\operatorname{num}}_{1,Vol}(\bU_j^h,\bU_l^h) 
+ \left( \mathbf{M}_{jj}  \mathbf{D}_{2,jl} +  \mathbf{D}_{2,lj} \mathbf{M}_{jj}  \right) \mathbf{f}^{\operatorname{num}}_{2,Vol}(\bU_j^h,\bU_l^h)  \right)\\
=& \sum_{j=1}^{n_p} \tau_j  (\mathbf{f}_j- \bfnum_j)
-\sum_{j,l=1}^{n_p} 
\left( \left( \mathbf{B}_{e_1,jl}  \right) \mathbf{f}^{\operatorname{num}}_{1,Vol}(\bU_j^h,\bU_l^h) 
+ \left( \mathbf{B}_{e_2,jl}  \right)\mathbf{f}^{\operatorname{num}}_{2,Vol}
(\bU_j^h,\bU_l^h) \right)
=
 \sum_{j=1}^{n_p} \tau_j  (\mathbf{f}_j- \bfnum_j)
-\sum_{j=1}^{n_p} 
  \tau_j  \mathbf{f}_j \\
  =& - \sum_{j=1}^{n_p} \tau_j \bfnum_j
 \end{aligned}
 \end{equation}}
and
\begin{equation}\label{eq:flux_two_dimension}
 \frac{\diff{}}{\diff t} \sum_{j=1}^{n_p} \omega_j \eta_j^h = \sum_{j=1}^{n_p} \left(\frac{1}{\omega_j}(\tau_{1,j} \psi_{1,j}  +\tau_{2,j}\psi_{2,j} -\tau_j \mathbf{w}_j^T \bfnum_j) \right)
 =\sum_{j=1}^{n_p} \tau_{j} \left(\psi_{\bn,j} -\mathbf{\vec{w}}_j^T \bfnum_j  \right).
\end{equation}
Since \eqref{eq:conservation} holds and  $\bfnum(\bU_j^{h,-}, \bU_j^{h,+},\bn^-)= -\bfnum(\bU_j^{h,+},\bU_j^{h,.},-\bn^-)$ cancels out, the scheme is local and global conservative. 
Due to \eqref{eq:flux_two_dimension} it is further local entropy conservative.  The entropy dissipation rate of the local Lax-Friedrich scheme at the interface point  $\bx_j$  is 
\begin{equation*}
 \tau_j \jump{\mathbf{w}^{h,T}_j} \cdot \bfnum(\bU_j^{h,-}, \bU_j^{h,+}, \bn^-) -\jump{\psi_j} \leq 0.
\end{equation*}
Summing all interface points together, the scheme is in total entropy dissipative. 
\end{proof}

For details on the high-order accuracy we refer to the above literature \cite[Theorem 4.1]{chenreview}. 
For the accuracy proof, we point out  that the  difference matrix $\mat{\mathbf{D}}$ is exact for polynomials of degree up to $n_p$ and $ \bfnum_{m, Vol}$ with $m=1,2$ is symmetric and consistent\footnote{The numerical volume flux $\bfnum (\bU_j^h,\bU_l^h)$ is  consistent with  $\mathbf{f}$ at least with $\Ol(h)$.}. Taking into account that $\bU^h \in \tilde{\VV}^h$ and $ \bfnum_{m, Vol}$ sufficiently smooth, we have 
\begin{equation}\label{approximiation_flux}
2\sum_{l=0}^{n_p} \mathbf{D}_{m,jl} \bfnum (\bU_j^h,\bU_l^h)-
\frac{\partial \mathbf{f}_m (\bU^h)}{\partial x_m} (\bx_j)=\Ol(h^{n_p})
\end{equation}
with a constant in $\Ol(h^{n_p})$ that depends on the regularity of $\mathbf{f}$ and $\mathbf{U}$. 
\begin{remark}[Extension to triangular grids, Gauss quadrature and further numerical fluxes]
We sum up:
\begin{itemize}
 \item It is possible to extend the  above results to triangular grids. Here, one can obtain  SBP operators if enough nodes are added at the cell interfaces. 
 \item Extensions using the Gauss-Legendre nodes or arbitrary volume or surface quadrature rules are also possible. This will lead to a more generalized SBP framework, cf. \cite{chenreview} and references therein.
 \item Instead of working with our numerical fluxes, i.e. Ranocha's flux \eqref{eq:Ranocha_flux} and the local Lax-Friedrich \eqref{eq:LLF}, we can use other fluxes like the entropy conservative flux of Chandrashekar for the volume part and the Godunov flux for the surface part.  In Section \ref{eq:numerics}, numerical tests are done also using  Chandrashekar's flux  to demonstrate the generality of our investigation.
\end{itemize}

\end{remark}

\subsection{Euler Equations - Weak BV Estimation}

We proceed by formulating the following assumption. Let $\rho^h(t), \bm^h(t), E^h(t):=\rho^h(t) \bu^h(t), E^h(t) \in \tilde{\VV}^h$ be numerical approximations of $\rho(t), \bm(t), E(t)$
obtained by our DG scheme \eqref{eq:Shu_modified_two_nodal}. 
 \begin{assumption}\label{assumtion_1}
  We assume that there exist two  positive constants $\underline{\rho}$ and $\overline{E}$ such that
\begin{equation}\label{vacuum}
\rho^h(t) \geq\underline{\rho} >0  \qquad E^h(t) \leq \overline{E} \text{ uniformly for } h\to 0.
\end{equation}
 \end{assumption}
The physical meaning of the first  assumption is that no vacuum appears. 
The second assumption \eqref{vacuum} implies then  that the speed $|\bu^h|$ is bounded since 
$
 |\bu^h|^2\leq \frac{2E^h}{\rho^h} \leq \frac{2\overline{E}}{\underline{\rho}}<C.
$
As it is described in   \cite{lukavcova2021convergence,feireisl2019convergence},  assumption \ref{assumtion_1} implies that  the density is also bounded from above and the energy is bounded from below.  Consequently, the pressure and temperature are bounded from above and below as well. 
Due to the application of bounded preserving limiters \eqref{se:extension}, we can get this property directly for all nodal values. 
In Section \ref{eq:numerics}, for the numerical experiments  we apply bounded preserving limiters and set the lowest value in the limiters to  $10^{-6}$. \\
Assumption \ref{assumtion_1}
is related also to the mathematical entropy function \eqref{eq:entropy_convex}. 
Indeed, it is  is equivalent to the strict convexity of the mathematical entropy function \eqref{eq:entropy_convex}. For completeness we recall the following Lemma from \cite[Lemma 3.1 and B2]{lukavcova2021convergence}: 
 \begin{lemma}
  Assumption  \ref{assumtion_1} is equivalent to the strictly positive definiteness of the the entropy Hessian, i.e. 
\begin{equation} \label{eq:hessian}
 \exists \underline{\eta}_0 >0: \frac{\diff{}^2\eta(\bU^h) }{\diff{\bU}^2} \geq \underline{\eta}_0 \mat{\mathcal{I}}
 \end{equation}
 where $\mathcal{I}$ is a unity  matrix.
 \end{lemma}
\begin{proof}
We demonstrate that  assumption \ref{assumtion_1} implies \eqref{eq:hessian}. We refer to \cite{lukavcova2021convergence} for the other direction of the proof. We suppress the dependence on $h$ in the following. 
 We calculate the Hessian (in our case for the two  dimensional case) using the entropy  and the entropy variables as defined in Section \ref{se:Foundation}.  Recalling that 
$
\eta=- \frac{\rho  s}{\gamma-1}$, $ s=\log \frac{p}{\rho^{\gamma}},
$
 the entropy variables are
$
\mathbf{w}=\eta'(\bU)= \left(\frac{\gamma}{\gamma-1}-\frac{s}{\gamma-1}-\frac{\rho |\mathbf{m}|^2}{2p}, \frac{m_1}{p},\frac{m_2}{p},
-\frac{\rho}{p} \right)^T, 
$
 $ p=(\gamma-1)(E- \frac{|\mathbf{m}|^2}{2 \rho})$.
Following Harten et al.  \cite{harten1983symmetric, harten1998convex}
we obtain the Hessian of $\eta$ with respect to the conservative variables $\bU=(\rho, m_1, m_2, E), \bm=(m_1,m_2)^T$:
$
 \frac{\diff{}\eta(\bU)^2 }{\diff{\bU}^2} =  \frac{\rho (\gamma-1)}{p^2} \mat{H} 
$
with the matrix 
\begin{equation}\label{eq:Hessian}
\mat{H}= \begin{pmatrix}
 \frac{1}{4} \left(\frac{|\bm|^2}{\rho^2}\right)^2 + \frac{\gamma}{(\gamma-1)^2} \frac{p^2}{\rho^2} & -\frac{m_1  |\bm|^2}{2 \rho^3} & -\frac{m_2  |\bm|^2}{2 \rho^3} &  \frac{1}{2} \frac{|\bm|^2}{\rho^2} - \frac{p}{\rho (\gamma-1)} \\
    -\frac{m_1  |\bm|^2}{2 \rho^3} & \frac{m_1^2}{\rho^2} + \frac{p}{\rho (\gamma-1)} &  \frac{m_1 m_2}{\rho^2} & - \frac{m_1}{\rho} \\
      -\frac{m_2  |\bm|^2}{2 \rho^3} &\frac{m_1 m_2}{\rho^2}  &  \frac{m_2^2}{\rho^2} + \frac{p}{\rho (\gamma-1)} 
      & - \frac{m_2}{\rho} \\
    \frac{1}{2} \frac{|\bm|^2}{\rho^2} - \frac{p}{\rho (\gamma-1)}  &  - \frac{m_1}{\rho}  
    &  - \frac{m_2}{\rho} & 1
\end{pmatrix}.
\end{equation}
Since $\gamma>1, \; \rho>0$ and $p>0$, we can demonstrate that $\mat{H}$ is positive definite. 
To this end,  we compute the determinates of  the major blocks of $\mat{H}$: 
\begin{equation*}
\begin{aligned}
M_{11}&=\mat{H}_{11}=  \frac{1}{4} \left(\frac{|\bm|^2}{\rho^2}\right)^2 + \frac{\gamma}{(\gamma-1)^2} \frac{p^2}{\rho^2} >0, \\
M_{22}&=\operatorname{det} \begin{pmatrix}
H_{11} & H_{12} \\
H_{21} & H_{22}
\end{pmatrix}
= \frac{p}{\gamma (\gamma-1)\rho}  \left\{  \left( \frac{ |\bm|^2}{2 \rho^2} -   \frac{\gamma}{\gamma-1} \frac{p}{\rho} \right)^2               +\gamma \frac{p m_1^2}{\rho^3}   + \frac{\gamma p m_2^2}{(\gamma-1) \rho^3} + \frac{(\gamma-1)|\bm|^4}{4 \rho^4}      \right\}
>0, \\
M_{33}&=\operatorname{det} \begin{pmatrix}
H_{11} & H_{12} & H_{13} \\
H_{21} & H_{22} & H_{23} \\
H_{31} & H_{32} & H_{33} 
\end{pmatrix}
=   \left(\frac{p^2}{(\gamma-1)^2 \rho^2} \right) \left(\frac{\gamma p}{(\gamma-1)\rho} \left(\frac{ |\bm|^2}{\rho^2} +  \frac{p}{(\gamma-1) \rho} \right)  + \frac{|\bm|^4}{4 \rho^4 }  \right) >0, \\
 M_{44}&=\operatorname{det} \mat{H}
=  \left( \frac{ p^4}{ \rho^4  (\gamma-1)^3}  \right)>0,
\end{aligned}
\end{equation*}
which implies that the Hessian is positive definite similar to \cite{harten1983symmetric}. \\
Next, we
demonstrate that the $\underline{\eta}_0>0$. We assume $\underline{\eta}_0=0$. Than, there is a subsequence $\{ \bU^{h_s }\}$ satisfying  $\frac{\diff{}\eta(\bU^{h_s})^2 }{\diff{\bU}^2} \leq \frac{1}{s}$. Since $\{\bU^{h_s} \}$ is bounded, we find again a subsequence (also denoted by $\bU^{h_s}$) which converges to some $\bU$. Hence, $\frac{\diff{}\eta(\bU)^2 }{\diff{\bU}^2}=0$
which is a contradiction. 
\end{proof}

 \begin{remark}
  Alternatively, we can focus on the eigenvalues of the Hessian as described in detail in \cite{lukavcova2021convergence}. We can estimate the lowest  eigenvalue by direct calculations (assuming $\gamma \in (1,2]$). 
It holds
\begin{equation*}
\underline{\eta}_0 \geq \frac{(\gamma-1)^2}{p} \min\left\{ \frac{p}{\rho (\gamma-1)( |\mathbf{u}|^2+2 )^2},   
 \frac{\rho (\gamma-1) }{4 ( p \gamma+\rho (\gamma-1) )}  , \frac{1}{4 \gamma (|\mathbf{u}|^2+1)} \right\} .
\end{equation*}
Assumption \ref{assumtion_1} yields  that $\underline{\eta}_0 $ is bounded away from zero. 
 \end{remark}

\subsubsection*{Weak BV estimate}
We denote by $\bU^h$ the unique solution of the DG scheme \eqref{eq:Shu_modified_two_nodal}  on the time interval $[0,T]$ with the initial data $\bU^h(0)$. 
We have seen that the entropy conservation/dissipation in one element is given as in 
\eqref{eq:flux_two_dimension}. Only the contribution at the boundary plays the role since the internal degrees of freedom cancel out 
and only the quadrature at the surface is essential. 
We denote with $\sigma$ the generic DOF on the interface between the two elements $K^-$ and $|K^+$.  It lies on the edge between $K^-$ and $|K^+$.
Due to our investigation in the proof 
of Theorem \ref{th:shu_chen_2}, we know that only the boundary terms remain.
Adding up the elements together, we find that 

\begin{align*}
\sum_{ K \in \TT_h} \frac{\dd } {\dd t} \left(  \frac{h^2}{4} \sum_{j=1}^{n_p} \omega_j \eta_j \right)& \leq \sum_{ K \in \TT_h}  \frac{h^2}{4} \sum_{j=1}^{n_p} \tau_{j}  \omega_{j}  \left(\psi_{\bn,j} -\mathbf{W}_j^{h,T} \bfnum_j  \right)\\
 &\leq \sum_{\sigma \in \EE} \frac{h}{2} \omega_{\sigma}   \left( \jump{(\psi_{\bn,\sigma})} -\jump{\mathbf{W}^h}_{\sigma} \bvecfnum_{\sigma}\right).
\end{align*}
$\omega_{\sigma}$ is the corresponding quadrature weight from the one-dimensional setting depending on the fact whether the boundary is in $x$ or $y$ direction\footnote{On each boundary, we have $p+1$ quadrature points (Gauss-Lobatto nodes). }.
We rewrite above equation as follows 
\begin{align}\label{eq_equa}
\int_{\Omega}  \frac{\dd } {\dd t} \eta( \bU^h(t)) \dd \bx -  \sum_{\sigma \in \EE} \frac{h}{2} \omega_{\sigma}   \left( \jump{(\psi_{\bn,\sigma})} -\jump{\mathbf{W}^h}_{\sigma} \bvecfnum_{\sigma}\right)\leq 0.
\end{align}
We  focus on the local Lax-Friedrich flux \eqref{eq:LLF} for the edge (surface) integrals (numerical flux). 
The entropy residual $ \left( \jump{(\psi_{\bn,\sigma})} -\jump{\mathbf{W}^h}_{\sigma} \bvecfnum_{\sigma}\right)$ in DOF $\sigma$ can be expressed as  
\begin{equation}\label{ent_residual}
r^h_\sigma= \delta_\sigma \jump{\bU^h}_{\sigma} \jump{\mathbf{W}^h}_{\sigma},
\end{equation}
where  $\delta_\sigma > \lambda_\sigma/2>0$, cf. \cite{feireisl2021numerics}. In case of the Euler equations, it holds that $\lambda_\sigma= \max\{ |\textbf{u}_{K^-}|+ c_{K^+},  |\textbf{u}_{K^+}|+ c_{K^+} \}$ with $c=\sqrt{\gamma \frac{p}{\rho}}$, respectively. 
The  speed of propagation is finite, i.e. there exists $\overline{\lambda}>0$ such that 
$\lambda(\bU^h(t))\leq \overline{\lambda}$ uniformly for $t\in [0,T]$ and $h \to 0$. 
Therefore, it follows directly by integrating in time 
\begin{align}\label{eq_equa}
 \int_{\Omega}  \eta( \bU^h(\tau)) \dd \bx - \int_0^\tau  \int_{\EE} r^h \dd S_{\sigma}  \leq  \int_{\Omega}  \eta( \bU^h(0)) \dd \bx .
\end{align}
With \eqref{eq:hessian} and the mean value theorem we get
\begin{equation*}
\jump{\bU}_{\sigma}= \left( \bU'(\tilde{\mathbf{W}})) \right) \jump{\mathbf{W}}_{\sigma} = \left(  \frac{\diff{}^2\eta(\tilde{\bU}) }{\diff{\bU}^2}\right)^{-1} \jump{\mathbf{W}}_{\sigma} ,
\end{equation*}
and thus
\begin{equation}
\underline{\eta}_0\jump{\bU^h}_{\sigma}\leq \jump{\mathbf{W}^h}_{\sigma}.
\end{equation}
Consequently, we have 
\begin{equation}\label{weak_BV}
\frac{\underline{\eta}_0}{2} \int_{0}^\tau  \sum_{\sigma \in \EE} \frac{h}{2} \omega_{\sigma}  \lambda_\sigma \jump{\bU^h}_{\sigma}^2 \diff t \leq
\int_{0}^T \sum_{\sigma \in \EE} \frac{h}{2} \omega_{\sigma}  \lambda_\sigma \jump{\bU^h}_{\sigma} \jump{\mathbf{W}}_{\sigma} 
\diff t.
\end{equation}
Now, it remains to demonstrate the weak BV condition. Using Hölder inequality, we get
\begin{equation}\label{weak_BV_calc}
\int_{0}^\tau   \sum_{\sigma \in \EE} \frac{h^2}{2}  \omega_{\sigma}   \lambda_\sigma   \left|\jump{\bU^h}_{\sigma} \right| \diff t  
\leq  \left( \int_{0}^\tau  \sum_{\sigma \in \EE} \frac{h^2}{2}  \omega_{\sigma}  \lambda_\sigma \diff t \right)^{1/2}
 \left( \int_{0}^\tau  \sum_{\sigma \in \EE} \frac{h^2}{2}  \omega_{\sigma}   \lambda_\sigma  \left|\jump{\bU^h}_{\sigma} \right|^2  \diff t \right)^{1/2}.
 \end{equation}
Inequalities \eqref{eq_equa} and \eqref{weak_BV}  imply that the second term tends to zero. 
Due to assumption \ref{assumtion_1}, $\lambda_\sigma$ is uniform bounded and  the discrete trace inequality:
$
\norm{f^h}_{L^p(\partial K)} \leq h^{-1/p} \norm{f^h}_{L^p(K)}, \;1\leq p\leq \infty
$
for any piecewise polynomial $f^h\in \tilde{\VV}^h$, cf.  \cite{feireisl2021numerics}, it follows directly that also the first term is bounded.
Consequently, 
we obtain the weak BV estimate:
\begin{equation}
 \int_{0}^\tau   \sum_{\sigma \in \EE}  \frac{h^2}{2} \omega_{\sigma}   \lambda_\sigma   \left|\jump{\bU^h}_{\sigma} \right| \diff t \to 0 \text{ as } h\to 0.  
\end{equation}

 \begin{remark}
 The above consideration holds also for the Godunov flux instead of the local Lax-Friedrich flux. 
 Indeed, the Godunov flux would lead to a different entropy residual  \eqref{ent_residual} but one obtains similarly weak BV estimates, cf. \cite{lukavcova2021convergence} for details.  
  \end{remark}

\section{Consistency of the DGSEM}\label{se:Consistency}
In the following, we prove the  consistency of our DG scheme \eqref{eq:Shu_modified_two_nodal}. 
Although the approximated problem was solved using the conservative variables $\bU^h=[\rho^h, \mathbf{m}^h, E^h]$, 
we investigate the consistency in terms of the conservative-entropy variables
$\CU^h=[\rho^h, \mathbf{m}^h, \eta^h]$ 
where $E^h=\frac{|\bm^h|^2}{2 \rho^h}+ e(\rho^h, \eta^h)$.
 Since the density and pressure are strictly 
positive, there is a bijective mapping between the conservative variables $\bU^h$ and the conservative-entropy variables $\CU^h$.
We will show that for the  numerical solution  $\CU^h=(\rho^h, \bm^h, \eta^h)$ calculated by  \eqref{eq:Shu_modified_two_nodal}
\begin{equation}\label{equ:consistent}
 \left[\int_{\Omega} \CU^h \cdot \varphi\diff \bx \right]_{t=0}^{t=\tau}=\int_0^\tau \int_{\Omega}  \partial_t \varphi \cdot  \CU^h + \bbf(\CU^h): \nabla_\bx \varphi \diff \bx \diff t +\int_0^\tau \mathbf{e}^h(t,\varphi) \diff t
\end{equation}
holds for all $\varphi \in C^{p+1}( [0,T]\otimes \overline{\Omega}, \R^4)$ where the error  $\mathbf{e}^h\to 0$ if $h\to 0$. We have to specify  $\mathbf{e}^h$ and in the following we describe the way to ensure \eqref{equ:consistent}.  \subsubsection*{Consistency Errors}
First, we realize that  for all $\varphi \in C^{p+1}( [0,T]\otimes \overline{\Omega}, \R^4)$
\begin{equation}\label{eq_consistent_2}
  \left[\int_{\Omega} \CU^h  \varphi\diff \bx \right]_{t=0}^{t=\tau}
  = \int_0^\tau \int_{\Omega} \frac{\diff{} }{\diff t} \left( \CU^h \varphi\right)\diff \bx \diff t 
  = \int_0^\tau \int_{\Omega} \CU^h \partial_t \varphi +  \varphi \partial_t \CU^h   \diff \bx\diff t
\end{equation}
and for the last term, we have to   apply the DG scheme \eqref{eq:Shu_modified_two_nodal} for $\partial_t \CU^h$ in tensor structure setting.
We have to determine the last terms of \eqref{equ:consistent}  using  \eqref{eq:Shu_modified_two_nodal} and derive a relation with equation \eqref{eq_consistent_2}. 
First, we note that in DG the test functions are taken from the space $\VV^h$ and not from $C^{p+1}$.
 We denote by $\Pi_h$ the projection  into our solution space  $ \VV^h$, i.e. projection on piecewise $\PP^p$ elements. We have the following interpolation errors then 
 \begin{equation}\label{interpolation_error}
 \norm{\Pi_h \varphi-\varphi}_{L^{\infty}(\Omega)} \leq c h^{p+1} \norm{\varphi}_{W^{p+1,\infty}(\Omega)}
 \end{equation}
with Sobolev norm $\norm{\cdot}_{W^{p+1,\infty}}$,  cf. \cite[Appendix]{calzado2015truncation}. 
We use the notations $\norm{\cdot}$, $\norm{\cdot}_2$ for the $L^1-$ and $L^2-$norm in the following.
With this, we have directly the  relation for the second term in \eqref{eq_consistent_2}:
\begin{equation*}
 \int_0^\tau \int_{\Omega}   \varphi \partial_t \CU^h   \diff \bx \diff t =
   \int_0^\tau \int_{\Omega}  \underbrace{\left(  \varphi-\Pi_h \varphi \right)}_{\Ol(h^{p+1})}  \partial_t \CU^h   \diff \bx \diff t 
   + \int_0^\tau  \int_{\Omega}  \left( \Pi_h \varphi \right)  \partial_t \CU^h   \diff \bx \diff t.
\end{equation*}
The first term is bounded by $\Ol(h^{p+1})$ and tends to zero for $h\to 0$. 
Using  DG formulation, e.g. \eqref{eq:DG_strong_analytical} with $  \bV=\Pi_h \varphi$ we get 
\begin{equation*}
\begin{aligned}
 \int_{\Omega} \partial_t \CU^h \Pi_h \varphi \diff \bx &+\sum_{K\in \TT_h} \int_K
  (\div_h \mathbf{f}^{vol}  (\CU^{h,-}, \CU^{h,+} )   )\Pi_h \varphi  \diff \bx \\
  &+ \sum_{\partial K\in \TT_h} \int_{\partial K^-} \left( \bfnum(\CU^{h,-}, \CU^{h,+}) - \mathbf{f}(\CU^h) \right)  \bn^-  \Pi_h \varphi^{-}
  \diff s=0,
\end{aligned}
\end{equation*}
where we have already applied discrete divergence on the numerical volume fluxes, 
i.e. $$\div_h \mathbf{f}^{vol} (\CU^{h,-}, \CU^{h,+} )  \sim \sum_{m=1}^2  \mat{\mathbf{D}}_{m} \bvecfnum_{m,S}(\bvu,\bvu).$$ 
We leave for a moment the question of numerical integration and nodal points to simplify the notation and calculation. We include them latter in numerical simulations. Recall, that we can use integration by parts since our operators are build to  fulfill the SBP property. 
We have
\begin{equation}
\sum_{K\in \TT_h} \int_K
  \div_h \mathbf{f}^{vol} \Pi_h \varphi  \diff \bx 
  =\underbrace{- \sum_{K\in \TT_h} \int_K \mathbf{f}^{vol} \nabla_h   
  (\Pi_h \varphi ) \diff \bx}_{T_1} + \underbrace{
  \sum_{\partial K\in \TT_h} \int_{\partial K^-}   \mathbf{f}^{vol} 
  \bn^-  \Pi_h \varphi^{-} \diff s}_{T_2}.
\end{equation}
First, we obtain
\begin{equation*}
T_1= -\underbrace{\sum_{K\in \TT_h} \int_K \left( \mathbf{f}^{vol} (\CU^{h,-}, \CU^{h,+} ) -\mathbf{f}(\CU^h) \right)\nabla_h   
  (\Pi_h \varphi ) \diff \bx }_{T_{11}} -\underbrace{
   \sum_{K\in \TT_h} \int_K\mathbf{f}(\CU^h) \nabla_h   
  (\Pi_h \varphi ) \diff \bx}_{T_{12}}.
\end{equation*}
Comparing above equation with \eqref{equ:consistent}, we want to keep $T_{12}$ . 
Since $\norm{\CU^h}_{L^{\infty}(\Omega)} \leq C$ and 
$$
\norm{\nabla_h \Pi_h \varphi -\nabla_\bx \varphi} \leq h \norm{\varphi}_{C^2(\Omega)},
$$
we have 
\begin{equation*}
\left|     \sum_{K\in \TT_h} \int_K\mathbf{f}(\CU^h) \nabla_h   
  (\Pi_h \varphi ) \diff \bx -    \sum_{K\in \TT_h} \int_K\mathbf{f}(\CU^h) \nabla_\bx  
   \varphi \diff \bx \right| \leq c h.
\end{equation*}
For $h \to 0$, this difference tends to zero and we obtain
\begin{equation}
 \sum_{K\in \TT_h} \int_K\mathbf{f}(\CU^h) \nabla_h   
  (\Pi_h \varphi ) \diff \bx \to    \int_ \Omega \mathbf{f}(\CU^h) \nabla_\bx  
   \varphi \diff \bx,
\end{equation}
which yield the flux term in \eqref{equ:consistent}. 
This means that the remaining terms need to be the error terms. 
To prove this,  we need an additional assumption expressing smoothness of a numerical solution inside an element $K$.  
 We assume that inside each element $K$ we have 
\begin{equation}\label{eq_assumption_2}
|\CU^h(\bx_j) - \CU^h(\bx_l)| \leq c h \qquad \forall \bx_j, \bx_l \in K. 
\end{equation}
Due to  Assumptions \ref{assumtion_1} and \eqref{eq_assumption_2}, and the Lipschitz continuity of $\mathbf{f}^{vol} (\CU^{h,-}, \CU^{h,+} )$, we get for $ T_{11}$:
\begin{equation*}
\begin{aligned}
\sum_{K\in \TT_h} \int_K \left( \mathbf{f}^{vol} (\CU^{h,-}, \CU^{h,+} ) -\mathbf{f}(\CU^h) \right)\nabla_h   
  (\Pi_h \varphi ) \diff \bx
  &\leq  c \norm{\varphi}_{C^2(\Omega)} \sum_{K\in \TT_h} \int_K
  \underbrace{\left|  \CU^{h,+} -  \CU^{h,-}\right|}_{\leq ch}  \diff \bx \\
  &\leq  c \norm{\varphi}_{C^2(\Omega)} |\Omega | h, 
  \end{aligned}
\end{equation*}
where $|\Omega | $ is the area of the domain. Consequently,  $T_{11} \to 0$ if $h\to 0$.  
We proceed with the boundary integrals: 
\begin{equation*}
  \sum_{\partial K\in \TT_h} \int_{\partial K^-}  \left( \mathbf{f}^{vol}(\CU^{h,-}, \CU^{h,+})  
 -\mathbf{f}(\CU^h) \right) \bn^-  \Pi_h \varphi^{-} \diff s+
 \sum_{\partial K\in \TT_h} \int_{\partial K^-} \bfnum(\CU^{h,-}, \CU^{h,+})    \bn^-  \Pi_h \varphi^{-}
  \diff s.
\end{equation*} 
The first flux, $ \mathbf{f}^{vol}$, is Ranocha's flux whereas the second flux, $\bfnum$ is the local Lax-Friedrich flux. 
We demonstrate in the following that all terms vanish under mesh refinement. 
Indeed, 
\begin{equation}\label{T_3_zero}
\begin{aligned}
  \sum_{\partial K\in \TT_h} \int_{\partial K^-}
  \mathbf{f}(\CU^h) \bn^-  \Pi_h \varphi^{-} \diff s
  =& \sum_{\partial K\in \TT_h} \int_{\partial K^-}
  \mathbf{f}(\CU^h)  \Pi_h \varphi^{-}  \bn^- \diff s\\
  =& \sum_{\partial K\in \TT_h} \int_{\partial K^-}
  \mathbf{f}(\CU^h)  \left(\Pi_h \varphi^{-}  -\varphi \right)\bn^- \diff s 
  + \sum_{\partial K\in \TT_h} \int_{\partial K^-}
  \mathbf{f}(\CU^h)  \varphi \bn^- \diff s .
\end{aligned}
\end{equation}
Due to  Assumptions \ref{assumtion_1}, we have $|\mathbf{f}( \CU^h)| \leq c$
and using \eqref{interpolation_error}, we obtain 
\begin{equation*}
 \sum_{\partial K\in \TT_h} \int_{\partial K^-}
  \mathbf{f}(\CU^h)  \left(\Pi_h \varphi^- -\varphi \right)\bn^- \diff s  \leq c h^p \underbrace{\sum_{\partial K\in \TT_h} \int_{\partial K^-} h}_{\leq c} = \Ol(h^p).
\end{equation*}
We get by using the Gauss theorem 
\begin{equation*}
\sum_{\partial K\in \TT_h} \int_{\partial K^-}
  \mathbf{f}(\CU^h)  \varphi \bn^- \diff s =
  \sum_{ K\in \TT_h} \int_{K}
  \div_{\bx}
  \left(\mathbf{f}(\CU^h)  \varphi \right)  \diff \bx
  = \int_{\Omega } \div_{\bx}
  \left(\mathbf{f}(\CU^h)  \varphi \right)  \diff \bx
  =\int_{\partial \Omega}  \underbrace{ \left(\mathbf{f}(\CU^h)  \varphi \right)}_{=0} \bn^-\diff s=0.
  \end{equation*}
  In the last step the periodic or no-flux boundary conditions have been applied. 
  The same consideration holds for both numerical fluxes. We
  have 
  \begin{equation*}
  \begin{aligned}
   &\sum_{\partial K\in \TT_h} \int_{\partial K^-} \bfnum(\CU^{h,-}, \CU^{h,+})    \bn^-  \Pi_h \varphi^{-}
  \diff s\\
   =&
   \sum_{\partial K\in \TT_h} \int_{\partial K^-}  \bfnum(\CU^{h,-}, \CU^{h,+})  \underbrace{\left(   \Pi_h \varphi^{-}- \varphi \right)}_{\Ol(h^{p+1})}\bn^-
  \diff s 
- \sum_{\partial K\in \TT_h} \underbrace{\int_{\partial K^-} \bfnum(\CU^{h,-}, \CU^{h,+})   \varphi \bn^-  \diff s}_{=0} =\Ol(h^p).
  \end{aligned}
  \end{equation*}
    Here, the second term vanishes due to the conservativity of the numerical fluxes and the boundary conditions. 
 We have shown that all remaining terms vanishes if $h\to 0$. 
 \begin{remark}[Including quadrature formulas]
  Till now, we have neglected the use of numerical integration and that $\mathbf{f}(\CU^h)$ is as well approximated by a polynomial $\mathbf{F}^h$ at the nodal Gauss-Lobatto points. Including these points lead to an additional quadrature error which vanishes under mesh refinement. To see this, we consider e.g. 
\begin{equation}\label{consisteny_estimate_1}
\begin{aligned}
\sum_{K \in \TT}\int_{K} \div \bbf(\CU^h) \Pi_h \varphi \diff \bx =
\sum_{K \in \TT}\int_{K} \div  \left(\bbf(\CU^h) -\mathbf{F}^h \right) \Pi_h \varphi \diff \bx + \sum_{K \in \TT}\int_{K} \div  \mathbf{F}^h \Pi_h \varphi \diff \bx \\
=\sum_{K \in \TT}  \frac{h^2}{4} \sum_{j=1}^{n_p} \omega_j \underbrace{ \left[ \div \left(\bbf(\CU^h) -\mathbf{F}^h \right) \Pi_h \varphi 
\right](\bx_j) }_{=0.}
+\sum_{K \in \TT} \frac{h^2}{4} \sum_{j=1}^{n_p} \omega_j \left[ \div  \mathbf{F}^h  \Pi_h \varphi 
\right](\bx_j) +\mathbf{e}^h_{qu}
\end{aligned}
\end{equation} 
In the last equation, we have used the Gauss-Lobatto quadrature rule where $\mathbf{e}^h$ denotes the quadrature error of the polynomial $\div \mathbf{F}^h  \Pi_h \varphi $. 
Actually, it is the difference between the exact integration and the discrete inner product. If the polynomial degree of the product is equal to or less than $2p-1$ in each direction, this error vanishes. Moreover, $\deg(  \div  \mathbf{F}^h  \Pi_h \varphi ) \leq 2p-1$ holds  in one direction. Therefore, we get an additional error term for each direction.
It is well-known that the remainder of the Gauss-Lobatto quadrature error for a given function $g\in C^\infty(I)$ with $I= (x_l,x_l+h)$ is determined by 
$R_n(g)=\frac{-p (p-1)^3((p-2)!)^4 }{(2p-1)((2p-2)!)^3} h^{2p-1}g^{2p-2}(\xi)$ with $\xi \in I$, cf. \cite[page 888]{abramowitz1972handbook}. Due the tensor structure setting, we can apply this result to smooth function $\mathbf{F}^h \Pi_h \varphi$. Consequently, the leading error term on the right hand side is  $\Ol(h^{2p-1})$. For $h\to 0$ and due to the boundedness of the remaining terms, we get
that $|\mathbf{e}^h_{qu}|\to 0$.\\
Finally, we study the first term \eqref{consisteny_estimate_1} which describes the truncation and aliasing errors since $
\div \bbf(\CU^h) \notin  \mathcal{Q}^p$ for arbitrary $\bx \in K$. However, we evaluate the functions on the quadrature points which are actually the interpolation points. 
Therefore, we have $\mathbf{f}(\CU^h)(\bx_j)=\mathbf{F}^h(\bx_j)$  and this term cancels out.
The above consideration can be brought to all of the remaining terms and we obtain our result.
 \end{remark}

 In summary, we have shown the consistency of the DGSEM method \eqref{eq:Shu_modified_two_nodal} for the Euler equation and summarize:

\begin{theorem}[Consistency Formulation]\label{th_consistency}
 Let $\CU^h$ be a solution of the DG scheme \eqref{eq:Shu_modified_two_nodal} on the interval $[0,T]$ with the initial data $\CU^h_0$. Under our assumptions \ref{assumtion_1} and \eqref{eq_assumption_2}, we have the following results for 
 all $\tau \in (0,T]$: 
 \begin{itemize}
  \item for all $\varphi \in C^{n_p+1}([0,T] \times \overline{\Omega})$:
  \begin{equation}\label{eq:consistency_rho}
 \left[  \int_{\Omega} \rho^h \varphi \diff \bx \right]_{t=0}^{t=\tau} =\int_0^\tau \int_{\Omega} \rho_h \partial_t \varphi +\bm^h \cdot \nabla_{\bx} \varphi \diff \bx \diff t
 +\int_0^\tau e_{\rho^h} (t,\varphi) \diff t;
  \end{equation}

  \item for all $\varphi \in C^{n_p+1}([0,T] \times \overline{\Omega};\R^2)$:
  \begin{equation}\label{eq:consistency_m}
 \left[  \int_{\Omega} \bm^h \mathbf{\varphi} \diff \bx \right]_{t=0}^{t=\tau} =\int_0^\tau \int_{\Omega} \bm^h \partial_t \mathbf{\varphi} +\frac{\bm^h\otimes\bm^h}{\rho^h} :  \nabla_{\bx}\mathbf{\varphi}  +p^h \div_{\bx} \mathbf{\varphi} \diff \bx \diff t + \int_0^\tau e_{\bm^h} (t,\mathbf{\varphi}) \diff t;
  \end{equation}
  
  \item for all $\varphi \in C^{n_p+1}([0,T] \times \overline{\Omega}), \; \varphi\geq 0$:
  \begin{equation}\label{eq:consistency_eta}
 \left[  \int_{\Omega} \eta^h \varphi \diff \bx \right]_{t=0}^{t=\tau} \leq \int_0^\tau \int_{\Omega} \eta^h \partial_t \varphi + (\eta^h \bu^h) \cdot \nabla_\bx \mathbf{\varphi} \diff \bx \diff t + \int_0^\tau e_{\eta^h} (t,\varphi) \diff t;
  \end{equation}
  \item
  $ \int_{\Omega} E^h(\tau)\diff \bx =\int_{\Omega} E^h_0 \diff \bx$
  \item The errors  $e_{j^h}$, $(j=\rho, \bm, \eta)$ tend to zero under mesh refinement 
  \begin{equation}\label{eq:consistency_error}
   \norm{e_{j^h}}_{L^1(0,T)} \to 0 \text{ if } h\to 0. 
  \end{equation}
 \end{itemize}
\end{theorem}


\begin{remark}[Consistency using subcell shock capturing]\label{re_subcell}
The pure consistency estimation of the DG framework has been proven under  assumption \eqref{eq_assumption_2}. 
However, in practice,  this is not needed if we apply the entropy stable FV subcell shock capturing technique proposed in \cite{hennemann2021provably} for the splitform DG formulation \eqref{eq:Shu_modified_two_nodal} on Gauss-Lobatto nodes. Here, the basic idea is to mark problematic cells where the discontinuity   may live, divide the cell into subcells considering every degree of freedom separately and apply finally in each subcell the underlying finite volume method. In our case, this would be the local Lax-Friedrich method. For the region with smooth solutions (not marked), we obtain error estimations
for the DG method following the spirit of \cite{huang2017error, zhang2006error} where adequate exactness of the quadrature rules of the volume and surface terms have been assumed. In such regions, the consistency property \eqref{equ:consistent} is clearly fulfilled. In the problematic cells which are effected by numerical oscillations driven by the Gibbs phenomena, the subcell FV method is applied. We obtain the consistency estimation for this part by the consistency investigation  from \cite{feireisl2019convergence} for the local Lax-Friedrich method. It has to be stressed out that the procedure of marking the problematic cells have a connection to our additional assumption if the gradient of $\bU^h$ is too steep, the cell will be marked. Due to our additional requirement,  we do not have such steep gradients inside our calculation, see also Remark \eqref{ex_discussion} for further discussion. 
\end{remark}

\begin{remark}[Discussion on the additional assumption \eqref{eq_assumption_2}]\label{ex_discussion}
We note that only a few results on the error behavior of DG schemes (or general high-order methods) are known in case that a nonsmooth solution is approximated. 
The closed one which is related to our setting is the result by Yang and Shu in  \cite{yang2013discontinuous} where the authors  investigated the analytical DG method. They proved for the scalar one-dimensional equation that the error behavior is high-order in smooth regions and the region around the discontinuities scales with $\Ol(h^{\frac{1}{2}}\log\frac{1}{h})$ where $h$ denotes the length of the mesh size. 
\end{remark}

\begin{remark}[Alternative surface and volume fluxes]

 Instead of working with the local Lax-Friedrich flux for the surface, we can apply any entropy stable monotone flux like 
Godunov flux for example. Following the analysis, in the FV framework from the literature \cite{lukavcova2021convergence}, we obtain similar estimations for those terms. 
Also, the volume flux can be changed to other entropy conservative fluxes similar results are obtained. 
\end{remark}

\section{Convergence to Dissipative Weak Solutions} \label{se:convergence}

In the following, we demonstrate the 
convergence of the entropy-stable DG method \eqref{eq:Shu_modified_two_nodal}.
 Due to our result from Section \ref{se:Consistency}, $\CU^h$  is a consistent 
approximation of the complete Euler system. In the following, we demonstrate the weak convergence theorem:

\begin{theorem}[Weak convergence]\label{eq:theorem_weak}
Let $\CU^h=\{\rho^h, \bm^h, \eta^h \}_{h\to 0}$ be a family of numerical solutions generated by the DG scheme \eqref{eq:Shu_modified_two_nodal} using the  chosen numerical fluxes as suggested in Section \ref{se:Consistency}.  Let assumptions \ref{assumtion_1} and  \eqref{eq_assumption_2} hold. Then, there exists a subsequence $\CU^h$ (denoted again by $\CU^h$) such that 
\begin{equation}\label{eq_bounds}
\begin{aligned}
\rho^h &\to \rho \text{ weakly-(*) in } L^{\infty}((0,T)\times \Omega)\\
\eta^h &\to \eta \text{ weakly-(*) in } L^{\infty}((0,T)\times \Omega)\\
\bm^h &\to \bm \text{ weakly-(*) in } L^{\infty}((0,T)\times \Omega; \R^2))\\
\end{aligned}
\end{equation}
as $h\to 0$, where $(\rho, \bm, \eta)$ is a DW solution of the complete Euler system \eqref{eq_Euler_conservation}. 
In addition, $\rho\geq \underline{\rho} >0$ and $\eta \leq  \underline{\rho s}$ a.a. in $(0,T)\times \Omega$.
Moreover, 
$E(\CU^h) \to \mean{E(\CU)}$ weakly-(*) in $L^{\infty}(0,T; \mathcal{M}(\overline{\Omega}))$ and the energy defect measure $\mathfrak{E}$ is a sum of the energy concentration defect $\mathfrak{E}_{cd}$ and the energy oscillation defect $\mathfrak{E}_{od}$. More precisely,  it holds
\begin{align*}
\mathfrak{E}_{cd} \equiv \overline{E(\rho, \bm, \eta)}- \est{\nu; E(\tilde{\rho}, \tilde{\bm}, \tilde{\eta})} =0,  \quad 
\mathfrak{E}_{od} \equiv  \est{\nu; E(\tilde{\rho}, \tilde{\bm}, \tilde{\eta})} - E( \rho, \bm, \eta) \geq 0. 
\end{align*}
Further, the momentum flux converges weakly-(*) in $L^{\infty}(0,T; \mathcal{M}(\overline{\Omega}; \R^{2\times2}_{sym}))$, i.e. 
\begin{equation*}
\left( \frac{\bm^h \otimes \bm^h}{\rho^h} +p(\rho^h, \eta^h) \mathbb{I}\right) \to \overline{ \frac{\bm \otimes \bm}{\rho} +p(\rho, \eta) \mathbb{I} }.
\end{equation*}
The Reynolds defect $\mathfrak{R}$ is a sum of the concentration defect $\mathfrak{R}_{cd}$ and the oscillation 
defect $\mathfrak{R}_{od}$.
We have 
\begin{align*}
\mathfrak{R}_{cd} \equiv& \overline{ \frac{\bm \otimes \bm}{\rho} +p(\rho, \eta) \mathbb{I} } -  \est{\nu ; \frac{\tilde{\bm} \otimes \tilde{\bm}}{\tilde{\rho}} +p(\tilde{\rho}, \tilde{\eta}) \mathbb{I} } =0, \\
\mathfrak{R}_{od}\equiv&  \est{\nu ; \frac{\tilde{\bm} \otimes \tilde{\bm}}{\tilde{\rho}} +p(\tilde{\rho}, \tilde{\eta}) \mathbb{I} } -\left( \frac{\bm \otimes \bm }{\rho} +p(\rho, \eta) \mathbb{I} \right).
\end{align*}
Specifically, the Reynolds defect is controlled by the energy defect in the following way
\begin{equation*}
c_1 \mathfrak{E} \leq tr[\mathfrak{R} ]\leq c_2 \mathfrak{E}
\end{equation*}
for certain constants $0<c_1 \leq c_2. $
\end{theorem}

\begin{proof}
 Under assumption \ref{assumtion_1} we obtain that 
 \begin{align*}
  \rho^{h} \in  L^{\infty}((0,T)\times \Omega), \; \bm^h \in  L^{\infty}((0,T)\times \Omega),\; \eta^h\in  L^{\infty}((0,T)\times \Omega),\;
  E^h \in  L^{\infty}((0,T)\times \Omega),\\
  \frac{\bm^h \otimes \bm^h}{\rho^h} \in  L^{\infty}((0,T)\times \Omega), \;p^h \in  L^{\infty}((0,T)\times \Omega),\; \mathbf{g}^h=\bu^h\eta^h \in  L^{\infty}((0,T)\times \Omega). 
 \end{align*}
In accordance with the fundamental theorem on  Young measures (cf. \cite{ball1989version}),
there exists a convergent subsequence and a paramtrized probability measure $\{ \nu_{t,\bx} \}_{(t,\bx)\in (0,T)\times \Omega } $ (in the context of \eqref{eq:Young}) ensuring that $\CU^h=(\rho^h,\bm^h, \eta^h)$ converges weakly-(*)  to 
$\est{\nu_{t,\bx},\tilde{\rho}},\;\est{\nu_{t,\bx},\tilde{\bm}},\; \est{\nu_{t,\bx},\tilde{\eta}}$ in $ L^{\infty}((0,T)\times \Omega)$. Moreover, also 
 $E^h \in  L^{\infty}((0,T)\times \Omega),\;
  \frac{\bm^h \otimes \bm^h}{\rho^h} , \;p^h,\; \mathbf{g}^h$  converges  weakly-(*)
  in the following sense in $L^\infty((0,T)\times \Omega)$:
\begin{equation*}
 \begin{aligned}
  E(\rho^h, \bm^h, \eta^h)&\to  \est{\nu; E(\tilde{\rho}, \tilde{\bm}, \tilde{\eta})}, \\
  \left( \frac{\bm^h \otimes \bm^h}{\rho^h} +p(\rho^h, \eta^h) \mathbb{I}\right) &\to \est{\nu ; \frac{\tilde{\bm} \otimes \tilde{\bm}}{\tilde{\rho}} +p(\tilde{\rho}, \tilde{\eta}) \mathbb{I} }.\\
 \end{aligned}
\end{equation*}
Next, we are passing to the limit $h\to 0$ in the consistency formulation from Theorem \ref{th_consistency}.
We get with \eqref{eq:consistency_rho}
\begin{equation*}
  \left[  \int_{\Omega} \est{\nu_{t,\bx}; \tilde{\rho}} \varphi \diff \bx \right]_{t=0}^{t=\tau} =\int_0^\tau \int_{\Omega} \est{\nu_{t,\bx}; \tilde{\rho}} \partial_t \varphi +\est{\nu_{t,\bx}; \tilde{\bm}}\cdot \nabla_{\bx} \varphi \diff \bx \diff t,
\end{equation*}
where  $\varphi \in C^\infty_c((0,T)\times \Omega)$. Analogously, we obtain similar results for
the momentum from \eqref{eq:consistency_m}
\begin{equation*}
  \left[  \int_{\Omega} \est{\nu_{t,\bx}; \tilde{\bm}} \varphi \diff \bx \right]_{t=0}^{t=\tau} =\int_0^\tau \int_{\Omega} \est{\nu_{t,\bx}; \tilde{\bm}} \partial_t \varphi +\est{\nu_{t,\bx}; \frac{\tilde{\bm}\otimes \tilde{\bm}}{\tilde{\rho}}}: \nabla_{\bx} \varphi \diff \bx \diff t
  +\int_0^\tau \int_{\Omega} \est{\nu_{t,\bx}; \tilde{p}} \div_{\bx} \varphi  \diff \bx \diff t,
\end{equation*}
for the entropy 
\begin{equation*}
\left[\int_{\Omega} \est{\nu_{t,\bx}; \tilde{\eta}}\varphi  \diff \bx \right]_{t=\tau_1-}^{t=\tau_2+} 
\leq  \int_{\tau_1}^{\tau_2} \int_{\Omega} \left[ 
\est{\nu_{t,\bx}; \tilde{\eta}} \partial_t \varphi +\est{\nu_{t,\bx}; \tilde{\eta}\tilde{\bm}} \cdot \nabla_\bx \varphi \right]  \diff \bx \diff t
\end{equation*}
 respectively (with suitable test functions $\varphi$ as specified in \eqref{eq:consistency_m}, \eqref{eq:consistency_eta}). 
With Theorem \ref{th_consistency}, we  get
  $$ \int_{\Omega} \est{\nu_{\tau,\bx}; E(\tilde{\rho},\tilde{\bm},\tilde{\eta})} \diff \bx =\int_{\Omega} E(\rho_0,\bm_0,\eta_0)\diff \bx.$$
This concludes that $\nu_{t,\bx}$ is a DMS satisfying \eqref{eq:Young}. Due to the minimum entropy principle and assumption \ref{assumtion_1}, we get further
\begin{equation*}
 \nu_{t, \bx}\left\{0<\underline{\rho}\leq \tilde{\rho}\leq\overline{\rho};\; \underline{s} \tilde{\rho} \leq (1-\gamma) \tilde{\eta} \leq \overline{s} \tilde{\rho} \right\}=1 \text{ for a. a. } (t,\bx) \in (0,T)\times \Omega; 
\end{equation*}

By identifying the first two coordinates of the barycenter of the Young measure with 
$
 \est{\nu; \tilde{\rho}}= \rho;\; \est{\nu; \tilde{\bm}}= \bm,
$
we have $\rho \in C_{weak}([0,T];L^{\gamma}(\Omega))$ and $\bm \in C_{weak}([0,T];L^{\frac{2\gamma}{\gamma+1}}(\Omega;\R^2) $. For  more details, cf. \cite[Section 2]{feireisl2021numerics} while $\eta=\est{\nu, \tilde{\eta}}$ satisfies 
$
  L^{\infty}(0, T; L^{\gamma}(\Omega)) \cap BV_{weak}([0,T];L^{\gamma}(\Omega)).
$
This show 
\eqref{eq_weak_space}.
From the limit process through the consistency properties of the entropy stable DG scheme  \eqref{eq:Shu_modified_two_nodal}  and the uniform bounds, we obtain that $(\rho, \bm, \eta)$ is a DW solution of the complete Euler system  as defined in Definition \ref{def_dmv}.
Following analogous steps as in \cite[Theorem 10.4]{feireisl2021numerics}, we 
obtain that the oscillation energy defect is positiv and the Reynolds defect can be controlled be the energy defect as describe in \cite[Theorem 10.4]{feireisl2021numerics}. 
\end{proof}

The weak convergence of the flux differencing method as described in Theorem \ref{eq:theorem_weak} is not very suitable in numerical simulations. Therefore, it is more convenient
to apply $\mathcal{K}$-convergence that provides strong convergence of the Cesaro averages  
to a DW solution as well as strong convergence of the approximate deviation of the associated Young measures, cf.  \cite[Theorem 10.5]{feireisl2021numerics}. 
In particular, we have \textbf{strong convergence of Cesaro averages} $\CU^{h_n}=(\rho^{h_n},\bm^{h_n},\eta^{h_n})$ meaning
\begin{equation*}
\frac{1}{N} \sum_{n=1}^N \CU^{h_n} \to \CU \text{ as } N\to \infty \text{ in } L^q ((0,T)\times \Omega, \R^4) \text{ for any } 1\leq q<\infty.
\end{equation*}
Under some additional assumptions as specified in Theorem \ref{th_convergence}, we can obtain strong convergence of the sequence of approximated solutions. 
Here, we can adapt again the proof  \cite[Theorem 10.6]{feireisl2021numerics} to our flux differencing scheme.
\begin{theorem}[Strong Convergence of the DG scheme]\label{th_convergence}
Let $\CU^h=\{ \rho^h, \bm^h, \eta^h \}_{h\to 0}$ be numerical solutions of DG method   \eqref{eq:Shu_modified_two_nodal}  with the initial data
$\rho^h_0, \bm_0^h$ and $\eta_0^h, \rho_0 \geq \underline{\rho}>0, (1-\gamma)\eta_0 \geq \underline{\rho s}$.
 Further, let assumptions \ref{assumtion_1} and  \eqref{eq_assumption_2} hold. 
Let $\CU^h=(\rho^h, \bm^h,\eta^h) \to (\rho, \bm,\eta) $ as $h\to 0$ 
in the sense specified in Theorem \ref{eq:theorem_weak}. Then, the following holds:
\begin{itemize}
\item \textbf{weak solution}: \\
If $\CU=[\rho, \bm, \eta]$ is a weak entropy solution and emanating from the initial data $\CU_0$, then 
$\nu_{t,\bx}=\delta_{\CU(t,\bx)} $ for a.a. $(t,\bx) \in (0,T) \times \Omega$, and 
\begin{align*}
(\rho^{h},\bm^h, \eta^h) &\to (\rho,\bm, \eta) \text{ in } L^q((0,T)\times \Omega;\R^4)\\
E(\CU^h)=\frac12 \frac{|\bm^{h}|^2}{\rho^{h}} + \rho^{h} e(\rho^{h}, \eta^{h}) &\to \frac12 \frac{|\bm|^2}{\rho} + \rho e(\rho, \eta)  \text{ in } L^q((0,T)\times \Omega)
\end{align*}
for any $1 \leq q <\infty$
\item \textbf{strong solution:}\\
Suppose that the Euler system admits a strong solution $\CU$ in the class 
$\rho, \eta \in W^{1,\infty}((0,T) \times \Omega), \bm \in W^{1,\infty}((0,T)\times \Omega; \R^2)$,
$\rho \geq \underline{\rho}>0$ in $[0,T]\times \Omega$  emanating from the initial data $\CU_0$. 
Then, for any $1\leq  q<\infty$ and $h\to 0$ 
\begin{align*}
(\rho^{h},\bm^h, \eta^h) &\to (\rho,\bm, \eta) \text{ in } L^q((0,T)\times \Omega; \R^4)\\
E(\CU^h)&\to E(\CU)  \text{ in } L^q((0,T)\times \Omega)
\end{align*}

\item \textbf{classical solutions:}\\
Let $\Omega\in \R^d$ be a bounded Lipschitz domain and
 $\rho\in C^1([0,T] \times \overline{\Omega})$, $\rho\geq \overline{\rho}>0,\; \bm\in C^1([0,T]\times \overline{\Omega};\R^2),\; \eta\in C^1 ([0,T] \times \overline{\Omega})$. Then $\CU=(\rho, \bm, \eta)$ is a classical solution to the Euler system and 
\begin{align*}
(\rho^{h},\bm^h, \eta^h) &\to (\rho,\bm, \eta) \text{ in } L^q((0,T)\times \Omega,\R^4)\\
\end{align*}
as $h\to 0$, for any $q\geq1$. 
\end{itemize}
\end{theorem}
\begin{proof}
[Sketch of the proof]
The defects  $\mathfrak{E}_{cd}$ and 
$\mathfrak{E}_{od}$ vanishes and the strong convergence of $E(\CU^{h_n}) \to E(\CU)$ in $L^q(0,T;L^1(\Omega))$ follows. Due to the sharp form of the Jensen inequality, cf. \cite[Lemma 7.1]{feireisl2021numerics}, we conclude that 
$\nu_{t,\bx}=\delta_{\CU} $ for a.a. $(t,\bx) \in (0,T)\times \Omega $  and in view of a priori bounds for our numerical solutions, we obtain the strong convergence of $\CU^{h}$ to a weak solution $\CU$. 
If the strong solution to the Euler system exists, we apply the weak-strong uniqueness principle, cf.  \cite[Theorem 6.2]{feireisl2021numerics}. 
Consequently, we have $\nu_{t,\bx}= \delta_{\CU},\; \mathfrak{R}=0 $, and $\CU$ is a strong solution. 
As the limit is unique, the whole sequence $\CU^h$ converges  strongly to the strong solution. \\
The last statement follows from the weak-strong uniqueness principle or directly from the compatibility property due to the regularity conditions of the classical solutions. 
\end{proof}

\begin{remark}[Extensions to other high-order schemes]\label{re_others}
Due to the close connection between our DG scheme to SBP-FD discretizations \cite{gassner2013skew}, the above considerations should also hold for FD schemes based on SBP techniques.
Besides this, other high-order FE discretizations can be applied as starting schemes. 
Here, we like to mention the general residual distribution (RD) framework including continuous Galerkin, SUPG and FR schemes \cite{abgrall2012review,abgrall2019reinterpretation} and the  invariant domain preserving
schemes \cite{guermond2018second, kuzmin2020monolithic}. An essential property is the consistency of a numerical scheme including entropy inequality.  
\end{remark}

\section{Numerical Simulations}\label{eq:numerics}

In this section, we focus on the Kelvin-Helmholtz problem to illustrate the weak, strong and $\mathcal{K}$-convergence of the DG method \eqref{eq:Shu_modified_two_nodal}. 
The domain is chosen to be $[0,1]\times[0,1]$ divided into $n\times n$ uniform quads. Denote the Cesaro average of the numerical solution 
$
 \tilde{\bU}^{h_n}=\frac{1}{n} \sum_{j=1}^n \bU^{h_j}. 
$
Let $\bU^{h_N}$ be the reference solution computed on the finest mesh with $N\times N$ elements. We compute two errors
\begin{equation*}
 E_1 =\norm{\bU^{h_n}-\bU^{h_N} }, \qquad E_2= \norm{\tilde{\bU}^{h_n} -\tilde{\bU}^{h_N}}.
\end{equation*}
$E_1$ gives the classical error behavior, whereas $E_2$ takes the Cesaro averages into account. 
The numerical test, we are considering is the famous Kelvin Helmholtz (KH)  problem \cite{kelvin1871hydrokinetic, holzholtz1868}. 
KH describes a shear flow of three fluid layers with different densities. 
The initial data are given by 
\begin{equation}\label{init_KH}
 (\rho, u, v, p)(\bx, 0)=\begin{cases}
                          (2,-0.5,0,2.5), \quad I_1 \leq y \leq I_2,\\
                          (1,0.5,0,2.5),  \quad \text{ otherwise 0},
                         \end{cases}
\end{equation}
where the interface profiles $I_j=I_j(\bx):= J_j+\epsilon Y_j(\bx),\; j=1,2$ are chosen to be small perturbations around the lower $J_1=0.25$ and the upper $J_2=0.75$ interfaces, respectively. Moreover, 
$$Y_j= \sum_{m=1}^M a_j^m \cos \left( b_j^m +2\pi m x \right), \quad j=1,2$$ with $a_j^m\in [0,1]$ and $b_j^m\in[-\pi, \pi],\; j=1,2,\; m=1,\cdots, M$ are arbitrary but fixed numbers. The coefficients $a_j^m$ have been normalized such that $\sum_{m=1}^M a_j^m=1$ to guarantee  that $|I_j-J_j|\leq \epsilon$ for $j=1,2.$ In the simulations, we have $M=10$, $\epsilon=0.01$ and $T=2$. We select for $\gamma=1.4$ and solve the Euler equation in conservative variables \eqref{eq_Euler_conservation} using our DG method \eqref{eq:Shu_modified_two_nodal}. For the volume flux, we select either Ranocha's flux \eqref{eq:Ranocha_flux}  or Chandrashekar's flux. The surface numerical flux is always the local Lax-Friedrich flux. To ensure the positivity of density and pressure, we apply always the limiters  by Shu and Zhang on every stage of the strong stability preserving Runge-Kutta (SSPRK) and set the lower bounds to $10^{-6}$.
As demonstrated in Theorem \ref{the_limiters}, the limiter does not increase the entropy and the DGSEM \eqref{eq:Shu_modified_two_nodal} is semidiscrete entropy dissipative. With the used SSPRK(10,4), we have not realized any violation of the entropy inequality in the fully discrete setting. All the implementations are done using the Trixi framework \cite{ranocha2021adaptive, schlottkelakemper2020trixi, schlottkelakemper2021purely}. Trixi is a powerful numerical simulation framework for hyperbolic conservation laws written in Julia and includes  all the above mentioned features inside. The time step size is also controlled by the Trixi internal time step control which calculates the maximum $\Delta t$ after each time step. The CFL constant is set to $1.3$. The experimental convergence study is done for the density, momentum and entropy. In the first part, we investigate everything without FV shock capturing. Afterwards, we include also the application of FV shock capturing as described in \cite{hennemann2021provably}. We will see that we obtain similar results. In Figure \ref{fig:init_approx2}, we plot the initial condition and the approximated solutions using different mesh sizes. 
\begin{figure}[tb]
	\centering 
	\begin{subfigure}[b]{0.32\textwidth}
		\includegraphics[width=\textwidth]{%
      		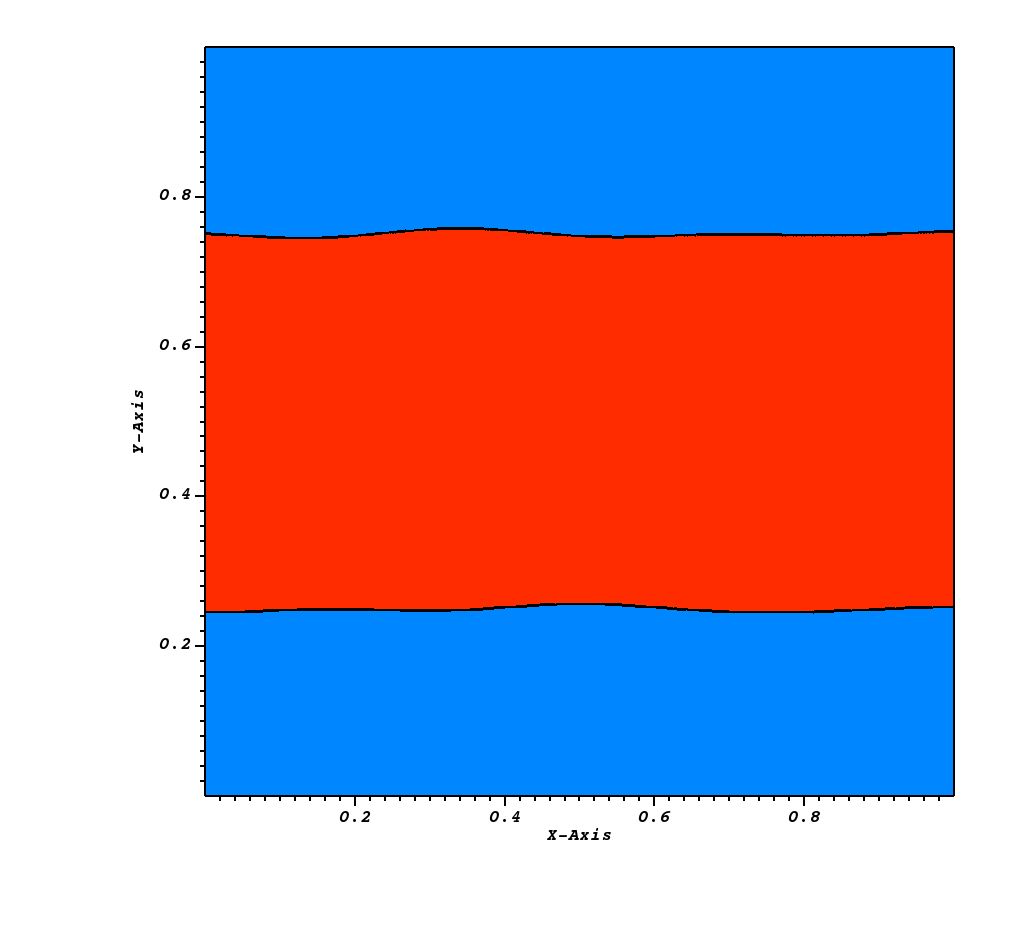} 
    		\caption{Initial Condition}
    		\label{fig:init_solution2}
  	\end{subfigure}%
	~
  	\begin{subfigure}[b]{0.32\textwidth}
		\includegraphics[width=\textwidth]{%
      		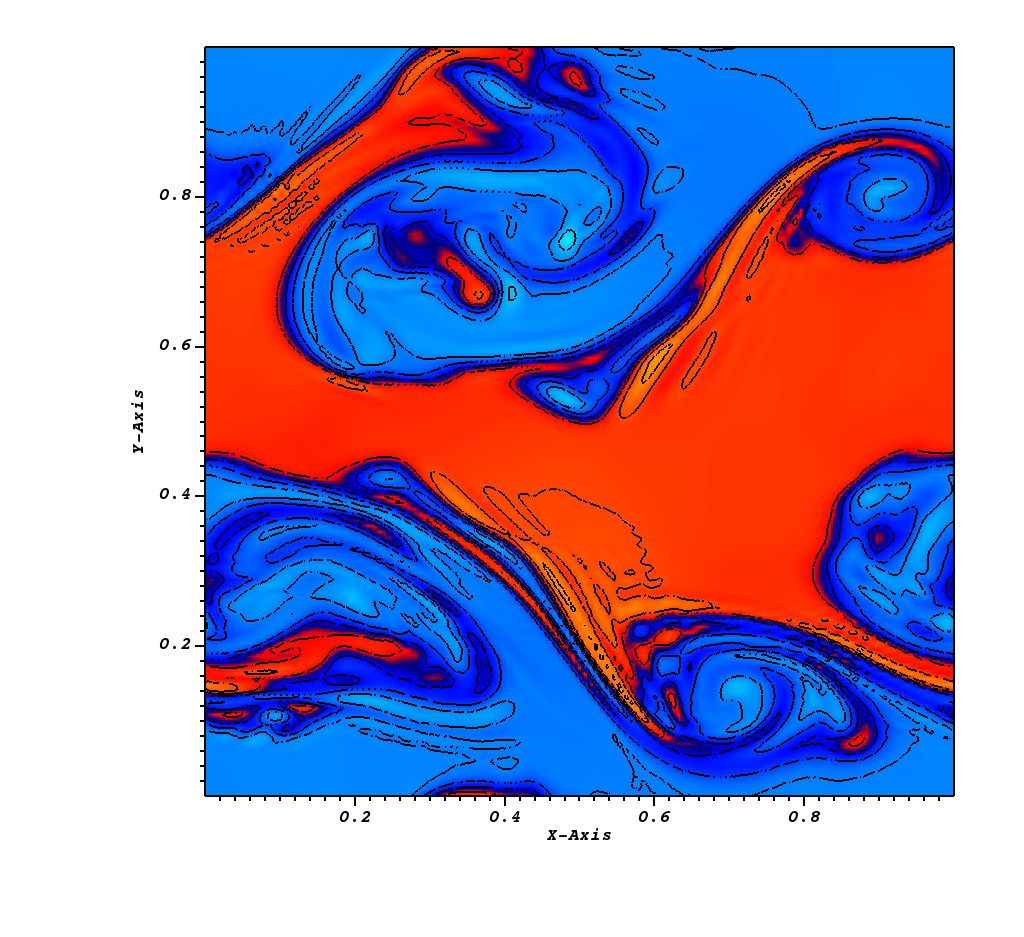} 
    		\caption{Numerical Solution with $256^2$ el.}
    		\label{fig:init_energy2}
  	\end{subfigure}%
	~
  	\begin{subfigure}[b]{0.32\textwidth}
		\includegraphics[width=\textwidth]{%
      		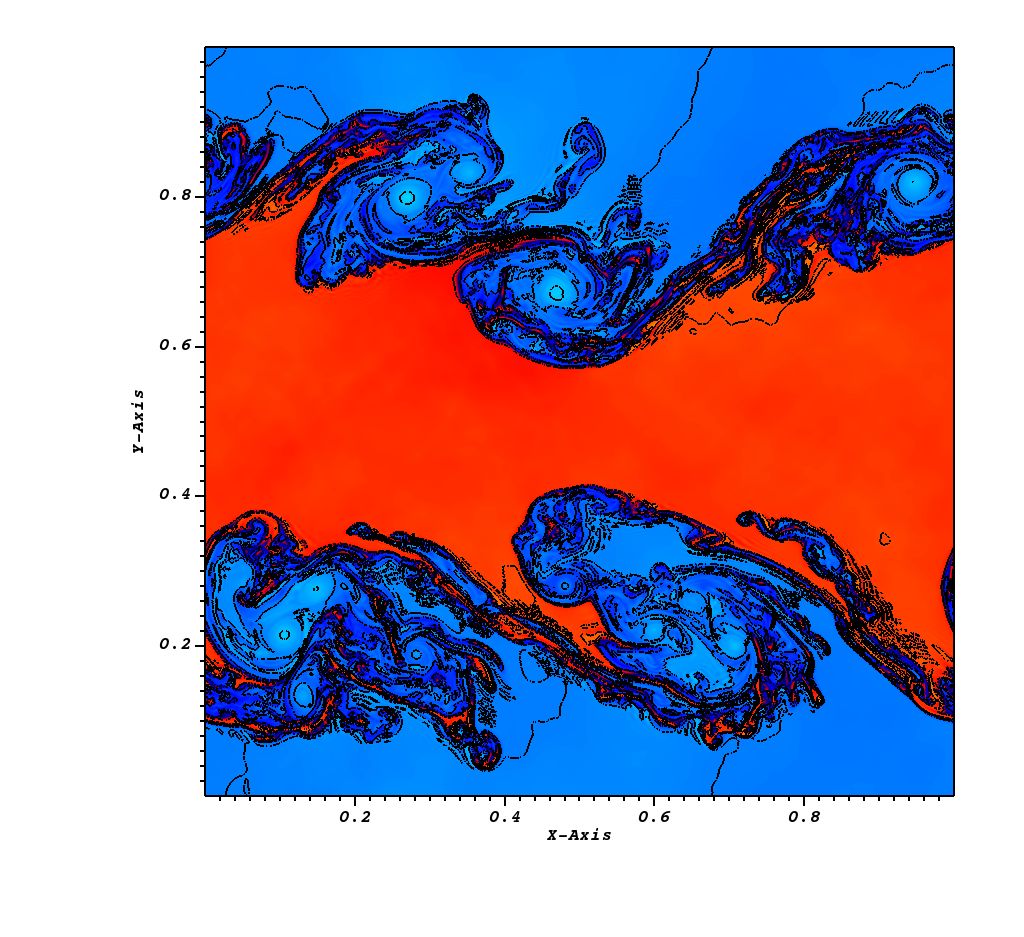} 
    		\caption{Numerical Solution with $2048^2$ el.}
    		\label{fig:init_ener}
	\end{subfigure}%
  	\caption{Kelvin-Helmholtz instability for different mesh sizes with FV shock capturing, $p=1$, $T=2$.}
  	\label{fig:init_approx2}
\end{figure}

\subsubsection*{DG without FV Shock Capturing}
We start our first investigation for $p=1$ elements resulting in a formally second order method for smooth solutions. However, KH  has already jumps in the initial condition \eqref{init_KH}. Therefore, a higher-order approximation can not be observed.  Numerical solutions are computed on subsequently refined meshes with $n\times n$ elements, where $n=32,\cdots, 4096$. 

In Table \ref{tab:table1}-\ref{tab:table3}, the  errors for the density, velocity field in $x$ direction, and the entropy are plotted. Obviously, the DG method \eqref{eq:Shu_modified_two_nodal} does not convergence strongly against the reference solution using $4096^2$ elements but we observe the convergence of the Cesaro averages in the $L^1$ norm as expected. The convergence rate tends to one in all variables. We recognize this also in our conter plots in Figure  \ref{fig:countour_density} and Figure \ref{fig:countour_mean}. In Figure \ref{fig:countour_density}, the numerical solutions using various meshes are plotted. 
We see that oscillations  are  developed and the structures in the numerical solutions using various meshes are different.  This is not the case if we look in Figure \ref{fig:countour_mean}, we plot the mean values where we start adding the numerical solutions of $512^2$ and $1024^2$  
 in Figure \ref{fig:countour_1_density}. In the following, we compute the Cesaro averages by adding
 one additional solution. A clear structure can be recognized and  convergence can be observed.  Similar results can be obtain for the momentum and entropy as well. For rest of this section, we restrict ourself to present only the density profiles. 
 \begin{table}[h!]
  \begin{center}
    \caption{$E_1$-$E_2$ error of the density and the corresponding order, $p=1$}
    \label{tab:table1}
    \begin{tabular}{l|c|c|c|c} 
      n & $E_1$-error & $E_1$-order & $E_2$-error & $E_2$-order\\ 
      \hline
      $32^2$& 0.3391987041589898 &- & 0.19661528297982866  & -\\
      $64^2$& 0.403851909363841& -0.25169570193630514 &  0.17980082465650782 &0.12897582947621833\\
      $128^2$& 0.41762701895495885&-0.04838869174988762 & 0.15735592207367408 & 0.1923681624534606  \\
      $256^2$& 0.3447284223874664 & 0.27675480022495547 &   0.1324612897928833 &  0.2484606642128492      \\ 
      $512^2$&0.24906184298174658&0.46895623852564644  & 0.09346227146130989 & 0.503114805826514 \\ 
      $1024^2$&0.21018735922716783 &0.24482810581029302 &0.05813201041147923 &  0.6850512972297058\\
      $2048^2$&0.1903225981357348 & 0.1432290359501281  &0.032795177858701306 & 0.8258491036824649   \\
    \end{tabular}
  \end{center}
\end{table}

  \begin{table}[h!]
  \begin{center}
    \caption{$E_1$-$E_2$ error of the momentum in $x$-direction and the corresponding order, $p=1$}
    \label{tab:table2}
    \begin{tabular}{l|c|c|c|c} 
      n & $E_1$-error & $E_1$-order & $E_2$-error & $E_2$-order\\ 
      \hline
      $32^2$&    0.22180097746995442 &   -                                 &       0.1014852261083653                         & -\\
      $64^2$& 0.296870662622228& -0.4205688068116397&    0.12277217792560473&     -0.2747139409280387\\
      $128^2$&   0.2812554275541182&  0.07795359309014603&   0.11021471582339173&  0.04130282973069437 \\
      $256^2$& 0.25917305217820485 & 0.11796521703383071 &       0.09463077351790287&    0.21993554165675272   \\ 
      $512^2$&  0.18835216910613375&   0.46048307233010094 &   0.06570345801497315&   0.5263401155398292\\ 
      $1024^2$&  0.1615674379789999&  0.22129617865208162&  0.039190252530856035&    0.7454744332203871\\
      $2048^2$&0.1713764028536377 &  -0.0850320065406901& 0.022004287330142883&          0.8327102224771352  \\
    \end{tabular}
  \end{center}
\end{table}

   \begin{table}[h!]
  \begin{center}
    \caption{$E_1$-$E_2$ error of the entropy and the corresponding order, $p=1$}
    \label{tab:table3}
    \begin{tabular}{l|c|c|c|c} 
      n & $E_1$-error & $E_1$-order & $E_2$-error & $E_2$-order\\ 
      \hline
      $32^2$&    0.8634607581918408 &   -                                 &        0.506355740513842                         & -\\
      $64^2$&  1.042908982694837&  -0.2724107376894449 &     0.4776001097892728&       0.08434813910457194\\
      $128^2$&  1.0385335649452683& 0.006065411504306775 &  0.4069706546364982&  0.230878398168195 \\
      $256^2$&  0.8368436093203517 &    0.3115179042930046 &     0.3336873275691747& 0.28642787169054407   \\ 
      $512^2$&  0.5772298357837078& 0.5358121631981844 &   0.23665851421583725&   0.4956900718266948\\ 
      $1024^2$&  0.49149138854289215&  0.23197973237718744 & 0.14708867505792783&  0.6861206545198556\\
      $2048^2$& 0.43815676117116004 &  0.1657190180182117 & 0.08185488762050568 &         0.8455457031863632  \\
    \end{tabular}
  \end{center}
\end{table}

   \begin{figure}[tb]
	\centering 
	\begin{subfigure}[b]{0.24\textwidth}
		\includegraphics[width=\textwidth]{%
      		 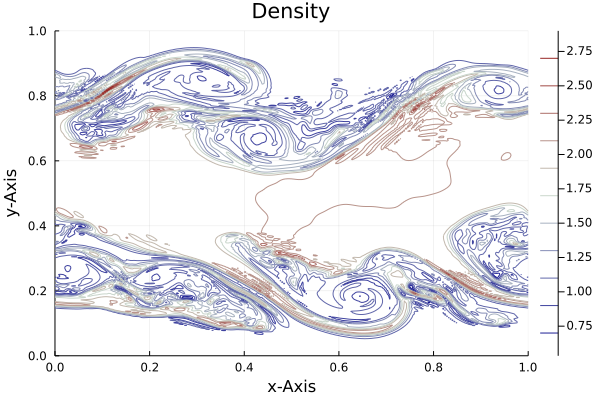} 
    		\caption{ $512^2$ el.}
    		\label{fig:countour_1_density}
  	\end{subfigure}%
	~
  	\begin{subfigure}[b]{0.24\textwidth}
		\includegraphics[width=\textwidth]{%
      		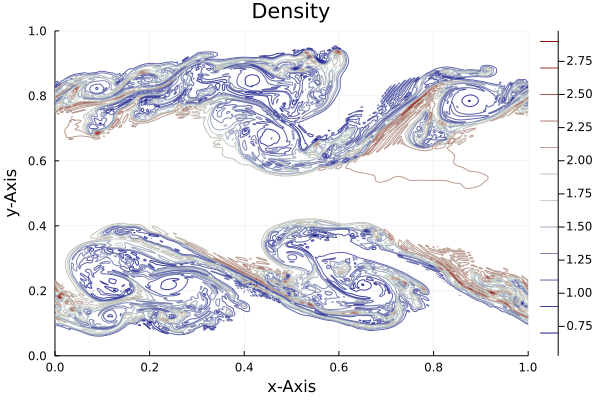} 
    		\caption{ $1024^2$ el.}
    		\label{fig:countour_2_density}
  	\end{subfigure}%
	~
  	\begin{subfigure}[b]{0.24\textwidth}
		\includegraphics[width=\textwidth]{%
      		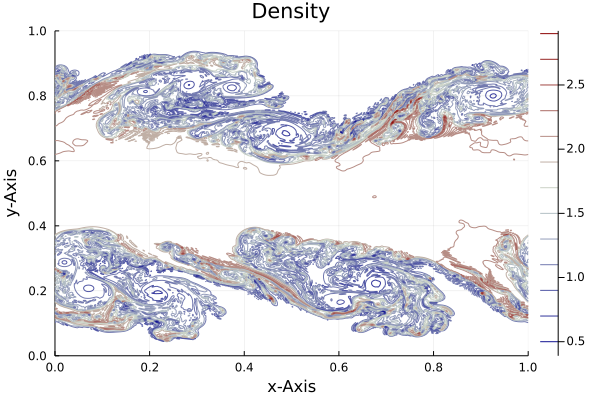} 
    		\caption{$2048^2$ el.}
    		\label{fig:countour_3_density}
	\end{subfigure}%
		~
  	\begin{subfigure}[b]{0.24\textwidth}
		\includegraphics[width=\textwidth]{%
      		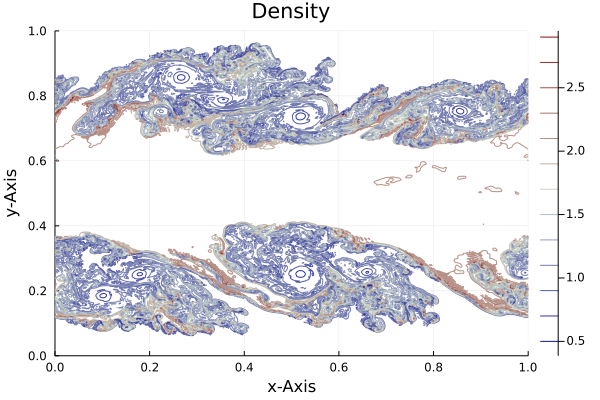} 
    		\caption{ $4096^2$ el.}
    		\label{fig:countour_4_density}
	\end{subfigure}%
  	\caption{Conter plot for the Kelvin-Helmholtz instability for different mesh sizes, $p=1$, $T=2$}
  	\label{fig:countour_density}
\end{figure}

   \begin{figure}[tb]
	\centering 
	\begin{subfigure}[b]{0.32\textwidth}
		\includegraphics[width=\textwidth]{%
      		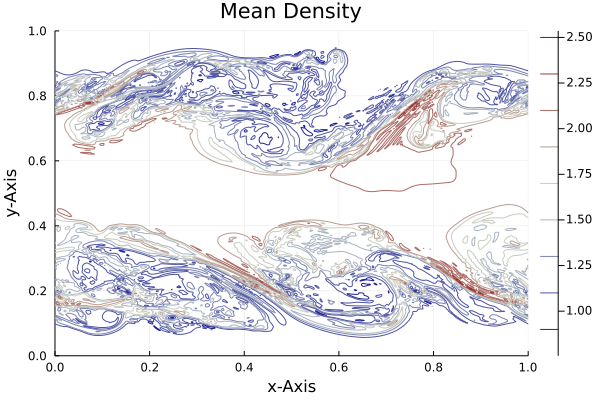} 
    		\caption{ Mean over two meshes}
    		\label{fig:countour_mean_1}
  	\end{subfigure}%
	~
  	\begin{subfigure}[b]{0.32\textwidth}
		\includegraphics[width=\textwidth]{%
      		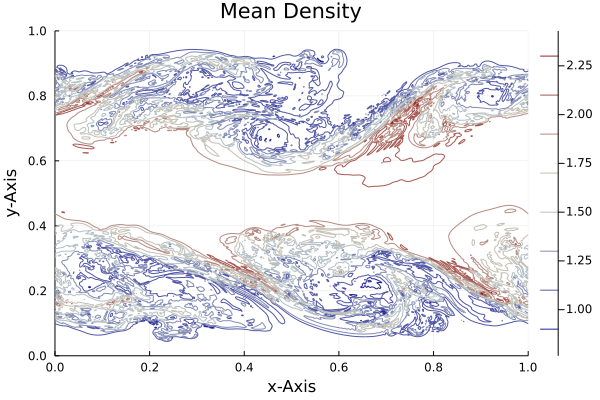} 
    		\caption{Mean over three meshes}
    		\label{fig:countour_mean_2}
  	\end{subfigure}%
	~
  	\begin{subfigure}[b]{0.32\textwidth}
		\includegraphics[width=\textwidth]{%
      		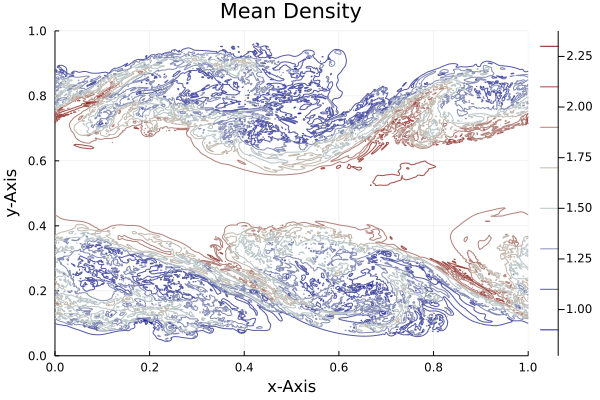} 
    		\caption{Mean over four meshes}
    		\label{fig:countour_mean_3}
	\end{subfigure}%
 	\caption{Conter plot of the means for the Kelvin-Helmholtz instability using various mesh sizes, $p=1$, $T=2$.}
  	\label{fig:countour_mean}
\end{figure}

In what follows, we select $p=2$. The reference solution is now obtained by using $2048 \times 2048$ elements. 
 In Figure \ref{tab:table4}, we present the error Cesaro averages of the density and the corresponding order  using Ranocha's flux and Chandrashekar's flux. 
 We observe that  nearly no difference between the convergence rates and the error can be recognized using Ranocha's flux and Chandrashekar's flux. However, a slide increase of the order can be recognized using $p=2$. 
Convergence rates for the momentum and entropy using the Cesaro averages are as well similar and are not presented here.
   \begin{table}[h!]
  \begin{center}
    \caption{$E_2$ error of the density and corresponding order using Ranocha's fluxes and Chandrashekar's fluxes, $p=2$ }
    \label{tab:table4}
    \begin{tabular}{l|c|c|c|c} 
      n & $E_2$-error & $E_2$-order & $E_2$-error with Ch.& $E_2$-order\ with Ch. \\
      \hline
      $32^2$ &     0.32360971460481636 &   -                                 &         0.32639198490486065                       & -\\
      $64^2$&   0.21635980762837606&    0.580822397551409 &    0.21777959203594555&       0.583736863474447\\
      $128^2$&   0.1693263231729843&  0.3536262503565837 &   0.16874256831080794&  0.36804480151192376 \\
      $256^2$&  0.11267169275425273 &    0.5876811661772481&      0.11359538566359653&  0.5709197324007521   \\ 
      $512^2$&  0.06561832191533114& 0.7799544985300991 &   0.06871005802380212&    0.7253110261242357\\ 
      $1024^2$&   0.03391358950826783&   0.9522352084021665 &  0.03551133314944197&    0.9522417790159925 \\
     \end{tabular}
  \end{center}
\end{table}

In Figure \ref{fig:error_withoutSC}, we plot the $E_2$ error behavior of the density using $p=2,3$ and different fluxes. We observe convergence of the Cesaro averages with a first order convergence rate.

In all of our simulations, we observe the convergence of the Cesaro averages but no strong mesh convergence of individual numerical solutions. 
Our results support  and verify the theoretical convergence results of Theorem \ref{th_convergence}.

\subsubsection*{ DG with FV Shock Capturing}
We illustrate the convergence of the DG method with subcell-limiting, cf. Remark \ref{re_subcell}. 
We start similar to above using $p=1$ elements going from $32\times 32$ to $2048\times 2048$ mesh cells with a reference solution on a mesh with $4096 \times 4096$ cells.
The setting of the shock sensors is $\alpha_{max}=0.002$ and $\alpha_{min}=0.0001$ in all calculations with the positivity preserving of density and pressure, cf. \cite{hennemann2021provably}.
In Table \ref{tab:table1_shock} we have plotted analogously to Table \ref{tab:table1} $E_1$ and $E_2$ errors and their rates.  

\begin{table}[h!]
  \begin{center}
    \caption{$E_1$-$E_2$ error of the density and the corresponding order with subcell FV limiting, $p=1$}
    \label{tab:table1_shock}
    \begin{tabular}{l|c|c|c|c} 
      n & $E_1$-error & $E_1$-order & $E_2$-error & $E_2$-order\\ 
      \hline
      $32^2$&  0.32620071294432 &- &  0.18620594884277916  & -\\
      $64^2$& 0.3826884539011542& -0.23041043888249668 &   0.1741552820199304 & 0.09652493473206789\\
      $128^2$& 0.39180560297402145&  -0.03396765476857858 & 0.15472760286112822& 0.17064363814240419  \\
      $256^2$& 0.32596094081182825& 0.2654389290439049 &    0.13150950951146823&  0.23456346626363495     \\ 
      $512^2$&0.24223750962649404&  0.42827682123120436 & 0.09411168346462777&  0.4827213831175634 \\ 
      $1024^2$&0.21100520733036873 & 0.19914367540003944 & 0.05669312945250009 & 0.7311999290713798\\
      $2048^2$& 0.19099215512416598 & 0.14376522144561604  & 0.03161301861015613 & 0.8426551084815536  \\
    \end{tabular}
  \end{center}
\end{table}
In Table \ref{tab:table2_shock} we present the errors of momentum in $x$-direction and entropy for the Cesaro averages. 
\begin{table}[h!]
  \begin{center}
    \caption{$E_2$ error of the momentum in $x$-direction and entropy,  and the corresponding order with subcell FV limiting, $p=1$}
    \label{tab:table2_shock}
    \begin{tabular}{l|c|c|c|c} 
      n & $E_2$-error mom. & $E_2$-order mom. & $E_2$-error entr. & $E_2$-order ent. \\ 
      \hline
      $32^2$&   0.09636282530552123 &- &  0.48202985147056043  & -\\
      $64^2$&  0.12086791867898412&    -0.3268827702770804&   0.45629006958259616 & 0.07917123722117338\\
      $128^2$&  0.11068505134027033&  0.12697097769624366 &  0.39354049881122233&  0.21343914669397857 \\
      $256^2$&  0.09727721901645832& 0.1862865005367386 &    0.3269830393966516&   0.267296304397868    \\ 
      $512^2$& 0.06759976270179612&  0.5250838031078213 & 0.2313402460238223&    0.4991995331272218 \\ 
      $1024^2$& 0.03906380943126804& 0.7911855370976559 &  0.14196346806488086&   0.7044965480368152\\
      $2048^2$& 0.022455417756270395 &    0.7987690830794113  &  0.07804379135369038 &  0.8631639533047681 \\
    \end{tabular}
  \end{center}
\end{table}
We see that the results are comparable to the results without subcell limiting. The subcell FV technique does not change the convergence result  to  dissipative weak solutions. Using $p>1$ and Chandrasheka's flux, we obtain analogous results to the first investigation, e.g. in Figure \ref{fig:error_withSC} we give the $E_2$ error plot for the density using $p=2,3$. 
\begin{figure}[tb]
	\centering 
	\begin{subfigure}[b]{0.4\textwidth}
		\includegraphics[width=\textwidth]{%
      		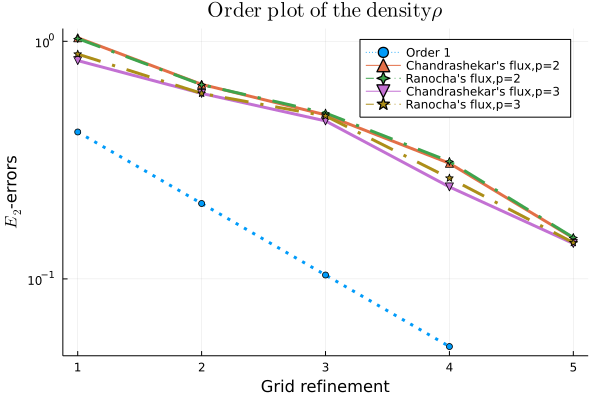} 
    		\caption{Without Subcell FV}
    		\label{fig:error_withoutSC}
  	\end{subfigure}%
	~
  	\begin{subfigure}[b]{0.4\textwidth}
		\includegraphics[width=\textwidth]{%
      		 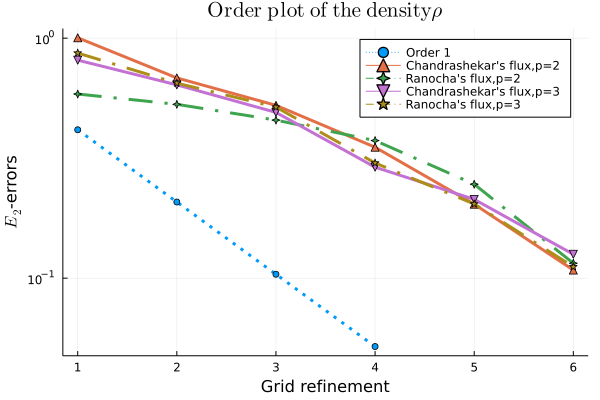} 
    		\caption{With Subcell FV}
    		\label{fig:error_withSC}
  	\end{subfigure}%
  	\caption{Error plots of the density using different Ranocha's flux and Chandrasheka's flux, $p=2,3$, $T=2$, dotted line is the reference first order slope.}
  	\label{fig:error_plots}
\end{figure}

\section{Conclusion}\label{se_con}

We have presented the first convergence analysis of the Euler equation via dissipative weak solutions for a high-order discontinuous Galerkin scheme. 
We have focused on entropy dissipative DG schemes using SBP operators.
A key point is that 
 scheme \eqref{eq:Shu_modified_two_nodal} is the  structure preserving properties and consistent with the underlying PDE, cf. Section \ref{se:Consistency}. We proved the  consistency by demanding  additional assumption \eqref{eq_assumption_2}.  In future, we plan to relax this condition and investigate the approximation properties of high-order schemes in the presence of shocks in more detail. However, other techniques to ensure consistency and  convergence can also be applied. Here, we used  subcell FV limiting to demonstrate that those techniques can also be applied in the context of DW solutions. The major key is that the limiting strategies are consistent in the presence of discontinuities.  Our investigation will be extended in the future taking into account other limiting strategies and further high-order schemes. Further,  the concept of dissipative weak solutions is not only restricted to the Euler equation but can also be used for magneto-hydrodynamics (MHD) and Navier-Stokes, cf. \cite{feireisl2021numerics, li2021convergence}.

\section{Appendix}\label{sec_appendix}

\subsection{Notation}\label{sub:notation}
In the following, we introduce some notations which are used in the main part of this work. 
In our notation, we follow \cite{feireisl2021numerics}
and denote by $\mathcal{M}^+(\overline{\Omega})$ the set of all nonnegative Borel measures on a topological space $\overline{\Omega}$. 
The symbol $\mathcal{M}(\overline{\Omega})$ denotes the set of all signed \textbf{Radon measures} that can be identified at the space of all linear forms on $C_c(\overline{\Omega})$, especially if $\overline{\Omega}$ is compact,
i.e. $[C_c(\overline{\Omega})]^*=\mathcal{M}(\overline{\Omega})$.
The symbol $\mathcal{M}^+(\overline{\Omega}; \R^{d\times d}_{sym})$ denotes the set of  \emph{positive semi-definite matrix valued measures}, i.e.
\begin{equation*}
\mathcal{M}^+(\overline{\Omega}, \R^{d\times d}_{sym} )= \left\{ \nu \in \mathcal{M}^+(\overline{\Omega}, \R^{d\times d}_{sym}) \big|
\int_{\overline{\Omega}} \phi(\xi \otimes\xi): \diff \nu\geq0  \text{ for any } \xi \in \R^d, \phi \in C_c(\overline{\Omega}), \phi\geq 0 
 \right\}.
\end{equation*}

A measure $\nu \in \mathcal{M}^+(\overline{\Omega})$ is called \emph{probability measure} if $\nu(\overline{\Omega}) =1$.
The (convex) set of all Borel probability measures on $\overline{\Omega}$ is denoted by $\mathcal{P}(\overline{\Omega})$. 
The symbol $C_{weak}(Q; X)$ denotes the space of functions on $Q$ ranging in a Banach space $X$, continuous with respect to the weak topology. More specific, $g \in C_{weak}(Q;X)$ if the mapping $y \to \norm{g(y)}_{X} $ is bounded and $y\to \est{f; g(y)}_{X^*,X}$ is continuous on Q for any linear from $f$ belonging to the dual space $X^*$.  
Next, we define defect measures. 
\begin{definition}[Defect Measure]\label{defect_measure}
Let $\Omega \subset \R^d$ be a bounded domain. Let $\{U_i\}_i^{\infty},( \norm{U_i}_{L^1(\Omega;\R^k} \leq c$),
be a sequence of functions generating a Young measure $\{ \nu_y \}_{y\in \Omega}$. Let $b\in C(\R^k)$,
$
|(b(V)| \leq c( 1+|V|), 
$
satisfies $b(U_i)\to \overline{b(U)}$ weakly-$(*)$ in $\mathcal{M}(\overline{\Omega})$. We call 
$
\overline{b(U)} - \left\{ y\to \est{\nu_y; b(\tilde{U})} \right\} \in \mathcal{M}(\overline{\Omega})
$
the \textbf{concentration defect} and the $\Omega$-measurable function 
$
\left\{ y\to \est{\nu_y; b(\tilde{U})} \right\}-b(U)
$
the \textbf{oscillation defect}.

\end{definition}
\begin{remark}
We recall the fundamental theorem of the theory of Young measures (cf. \cite{ball1989version}), that states
\begin{equation*}
\est{\nu_{t,x}, g(U)} = \overline{g(U)}(t,x) \text{ for a.a. } (t,x) \in (0,T)\times \Omega, 
\end{equation*}
whenever $g \in C_c(\mathcal{F})$ and 
\begin{equation*}
 g(U^h) \wwto \overline{g(U)}  \text{ in }  L^1((0,T)\times \Omega). 
\end{equation*}
\end{remark}

\subsection{Connection with Summation-by-parts Operators}\label{eq:SBP property}
 Here, we describe the SBP operators and repeat some well-known properties following 
   \cite{chenreview, offner2020stability}. Since we have used Gauss-Lobatto points for each direction in tensor product form, we restrict ourselves to the one dimensional setting to demonstrate the SBP property.

\begin{theorem}[Summation-by-parts Property]\label{Th:SBP_Lobatto}
The one-dimensional operators\footnote{We omit the index $1$ in the following proofs.} from Section \ref{se:flux_differencing} fulfill the SBP property 
\begin{equation}\label{eq:SBP_property_DG}
\mat{Q}+\mat{Q}^T=\mat{M}\mat{D}  + (\mat{MD})^T  = \mat{B}=\diag(-1,0,\cdots,0,1).
\end{equation}

\end{theorem}
\begin{proof}
It is  $Q_{jl}=M_{jj}D_{jl}=\omega_j L_{l}'(\xi_j) =
 \sum_{j=0}^p \omega_r L_j (\xi_r) L_{l}'(\xi_r)
 =\est{L_j,L_l'}_{\omega}$
and so 
\begin{equation*}
M_{jj}D_{jl}+D_{lj}M_{jj}=\est{L_j,L_l'}_{\mat{M}}+\est{L_l,L_j'}_{\mat{M}}= L_j(1)L_l(1)-L_j(-1)L_l(-1)=
\delta_{pj}\delta_{pl}-\delta_{0j}\delta_{0l}.
\end{equation*}
Hence,
$\mat{B}=\mat{M}\mat{D}  + (\mat{MD})^T$.
\end{proof}
Another useful relations are  following.
\begin{proposition}\label{prop:derivative}
For each $0\leq j\leq p$ we have
\begin{equation}
\sum_{l=0}^p D_{jl}=\sum_{l=0}^p Q_{jl} =0, \qquad 
\sum_{l=0}^p Q_{lj}=\tau_j=\begin{cases}
-1 \quad j=0\\
1 \quad j=p\\
0 \quad 1\leq j\leq p-1
\end{cases}.
\end{equation}
\end{proposition}
\begin{proof}
Since the sum of the Lagrange polynomials are $\sum_{l=0}^pL_l(\xi)=1$, we get
\begin{equation*}
\begin{aligned}
\sum_{l=0}^p D_{jl}=\sum_{l=0}^p L'_l(\xi_j)=0, \qquad 
\sum_{l=0}^p Q_{jl}=w_j\sum_{l=0}^pD_{jl}=0,\\
\sum_{l=0}^p Q_{lj}=\sum_{l=0}^p B_{jl}-\sum_{l=0}^p Q_{jl}=
\sum_{l=0}^p B_{jl}= B_{jj}=\tau_j.
\end{aligned}
\end{equation*}
\end{proof}
Due to the tensor structure, the results transform automatically to the two dimensional setting.
 
\subsection{Extension to the Fully Discrete Setting - Techniques in the Implementation}\label{se:extension}

Our DG method \eqref{eq:Shu_modified_two_nodal} needs to guarantee 
that the density and internal energy (consequently pressure and temperature) remain positive at all degrees of freedom on $(0,T)$. 
To obtain this, we apply the well-known limiting strategy  by Zhang and Shu  \cite{zhang2011maximum}. As written in \cite{zhang2011maximum}, the approach can be summarized as follows\footnote{We omit the $h$-dependence in the following.}. To construct high-order DG (or WENO) schemes preserving the positivity of density and pressure, we need essential those four steps:
\begin{enumerate}
 \item Introduce the admissible set 
 \begin{equation}
  G= \left\{  \bU=\begin{pmatrix}
    \rho \\
        \bm\\
        E
      \end{pmatrix} \big| \rho>0  \text{ and } p=(\rho-1) \left( E- \frac{1}{2} \frac{|\bm|^2}{\rho}\right)  >0  \right\}
 \end{equation}
 and prove that $G$ is a convex set. 
 \item Focus on the chosen first order scheme for the Euler equation and prove that it preserves the positivity of $\rho$ and $p$.
  \item Find a sufficient condition for the Euler forward time discretization that the cell averages $\mean{u}$ of the DG method (and so for the underlying first order scheme) 
 remain in $G$. Then, high order strong stability preserving Runge-Kutta (SSPRK) method will keep the positivity due to the convexity of $G$.
 \item Construct and apply a limiter to enforce the positivity at the nodal values.
\end{enumerate}
It is shown in \cite{zhang2012minimum} that $G$ is  a convex set. 
The first order method in our case reduces to a simple cell centered FV method.
The properties of the surface fluxes are essential for the  basic properties of the scheme. It is well-known that both the (local) Lax-Friedrich as well as the Godunov 
fluxes are  positivity preserving fluxes. So both preserve invariant domains under suitable time step restriction. Again, here  assumption \ref{assumtion_1}
is needed to ensure those properties. \\
As it is described \textit{inter alia} in \cite{zhang2011maximum}, for rectangular meshes the CFL condition is given by
 $\lambda_x a_x +\lambda_y a_y \leq \alpha_0$ where the term $\lambda_m$ is the ratio of time and space mesh size  $\lambda_m= \frac{\Delta t}{\Delta x/\Delta y}$
and $a_m$ the maximum speeds\footnote{$u_m+c$ comes from the eigenvalues of the Jacobi matrices of the Euler fluxes.} $a_m= \norm{|u_m|+c}_{\infty}$.
Using SSPRK method will transfer the result to higher order in time. The cell averages of our DG method remain in the set $G$. Finally, the limiters are applied
to guarantee not only the positivity for the cell averages but also for the nodal values. The limiter itself is a simple linear scaling procedure 
$\tilde{\bU}_j^{i,n} =\mean{\bU}_j^{i,n} +\theta_j^{i,n} ( \bU_j^{i,n}-\mean{\bU}_j^{i,n} )$ to enforce $\tilde{\bU}_j^{i,n} \in G$ which is always possible as long as $\mean{\bU}_j^{i,n} \in G$. Roughly speaking, for each $j$, we compute 
$\theta_j^{i,n}=\max\{s \in [0,1]| \mean{\bU}_j^{i,n} +\theta_j^{i,n} ( \bU_j^{i,n}-\mean{\bU}_j^{i,n}) \in G ) \}$. Then, we simply let $\theta^{i,n}=\min_j \theta_j^{i,n}$. 
The limiter tends to squeeze the data towards the cell average. 
 The limiting process returns us nodal values satisfying the conservation property but 
 enforces the pressure and density values to be positive. It has to be used at each time step. Furthermore, due to assumption  \ref{assumtion_1} on the cell average, we obtain analogous results about the boundedness of the nodal values by the application of the limiters. Therefore, not only the averages of the variables are bounded but also their nodal values.  \\
It is important that the bound-preserving limiter is compatible with the entropy inequality and it actually does not increase the entropy, cf. \cite{chen2017entropy}. 

\begin{theorem}\label{the_limiters}
Suppose $\alpha_j$, $\bU_j \in \Omega$ for $ 1\leq j\leq n_p$ with $\sum_{j=1}^{n_p} \alpha_j =1$. Define the average $\mean{\bU}= \sum_{j=1}^{n_p} \alpha_j \bU_j$.
We modify these values without changing the average. That is, let $ \tilde{\bU}_j=\mean{\bU} + \theta_j (\bU_j-\mean{\bU})$ such that $0\leq \theta_j \leq 1$ and
$\mean{\bU}= \sum_{j=1}^{n_p} \alpha_j \tilde{\bU}_j$. Then for any convex entropy function $\eta$, we have 
$\sum_{j=1}^{n_p} \alpha_j \eta(\tilde{\bU}_j) \leq \sum_{j=1}^{n_p} \alpha_j \eta(\bU_j)$. The limiter does not increase the entropy. 
\end{theorem}

\begin{proof}
Since $\sum_{j=1}^{n_p} \alpha_j \tilde{\bU}_j= \sum_{j=1}^{n_p} \alpha_j\left( \mean{\bU} + \theta_j (\bU_j-\mean{\bU}) \right) =\mean{\bU}$, we have
$
\sum_{j=1}^{n_p} \alpha_j(1-\theta_j) \bU_j = \left(\sum_{j=1}^{n_p} \alpha_j(1-\theta_j) \right) \mean{\bU}.
$
Due to the convexity of $\eta$, we get 
$
\eta(\tilde{\bU}_j) \leq \theta_j \eta(\bU_j) +(1-\theta_j)  \eta(\mean{\bU})$ and so $$
 \left(\sum_{j=1}^{n_p} \alpha_j(1-\theta_j) \right) \eta(\mean{\bU}) \leq \sum_{j=1}^{n_p} \alpha_j(1-\theta_j) \eta(\bU_j).
$$Hence, 
\begin{align*}
\sum_{j=1}^{n_p} \alpha_j\eta(\tilde{\bU}_j) &\leq \sum_{j=1}^{n_p} \alpha_j\left(  \theta_j \eta(\bU_j) +(1-\theta_j)  \eta(\mean{\bU})\right)
= \sum_{j=1}^{n_p} \alpha_j  \theta_j \eta(\bU_j)+ \left( \sum_{j=1}^{n_p} \alpha_j (1-\theta_j)\right) \eta(\mean{\bU})\\
&\leq \sum_{j=1}^{n_p} \alpha_j  \theta_j \eta(\bU_j)+ \sum_{j=1}^{n_p} \alpha_j (1-\theta_j)\eta(\bU_j) =\sum_{j=1}^{n_p} \alpha_j \eta(\bU_j). 
\end{align*}
Therefore, the limiter does not increase the entropy. 
\end{proof}
As described in  \cite{zhang2011maximum}, the limiter is based on the  Gauss-Lobatto nodes and does not violate our entropy condition \eqref{iq:entropy}. 
Finally, we stress out that a minimum entropy principle will also be satisfied by extending $G$ using the specific convex entropy. Therefore, in the calculation of 
$\theta$ this has to be taken into account. For a detailed explanation as well as the implementation  details we refer again to \cite{zhang2011maximum}. \\
For our consideration, it is enough to focus on simple underlying FV scheme and so the cell averages since by applying those limiters we can extend the results to DGSEM schemes.

 \begin{con}\label{con_DG_Space}
If we apply DG method \eqref{eq:Shu_modified_two_nodal}  with Ranocha's flux \eqref{eq:Ranocha_flux} for the volume part  and 
some basic numerical flux like (local) Lax-Friedrich flux (or Godunov) for the surface integral, we derive  a high order scheme that is conservative and entropy stable in the semidiscrete setting. By applying the bounded preserving limiter from Zhang and Shu at each  time step $t^n$, we can further 
guarantee that the  scheme is positivity preserving.
Further, assumption \ref{assumtion_1} is ensured also for nodal solution values due to the construction of the limiter.
\end{con} 

To ensure that the discrete scheme is  entropy stable, various techniques can be applied.\\
In the FV framework in  \cite{feireisl2019convergence}, implicit schemes are only considered for the time-integration. We can apply implicit SSPRK schemes which should ensure the physical bounds and  high order accuracy. 
 Another way is to use the relaxation approach \cite{ranocha2020relaxation}.  Here, the main idea is to apply a classical time-integration scheme (in our case explicit SSPRK schemes) and adapt the last subtimestep in such a way that we obtain entropy dissipation. 
 
 \begin{con}
Combining our limited DG method with the relaxation approach \cite{ranocha2020relaxation}, we obtain a positivity preserving, entropy dissipative, conservative and arbitrary high order DG scheme. 
\end{con}

\section*{Acknowledgements} 
M.L.-M. has been founded by the German Science Foundation (DFG) under the collaborative research projects TRR SFB 165 (Project A2) and TRR SFB 146 (Project C5).\\
M.L.-M. and P.\"O. gratefully acknowledge support of the Gutenberg Research College,  JGU Mainz.

\small
\bibliographystyle{abbrv}
\bibliography{literature}

\end{document}